\documentclass[11pt, a4paper]{amsart}

\usepackage{amssymb, bbm}
\usepackage{chngcntr}
\usepackage[utf8]{inputenc}
\usepackage{tikz}
\usepackage{subcaption}
\usepackage[normalem]{ulem}

\usepackage{fancyhdr}
\usepackage{amsaddr}
\usepackage{color}
\usepackage{enumitem}
\usepackage{dsfont}
\usepackage{cleveref}
\usepackage{xargs}
\usepackage{tikz}
\usetikzlibrary{calc,shapes,snakes,patterns,matrix,backgrounds,trees,arrows}
\usepackage{algorithm}
\usepackage{algpseudocode}

\usepackage[sort, nocompress]{cite}

\usepackage[margin=0.75in]{geometry}

\newtheorem{theorem}{Theorem}[section]

\newtheorem{lemma}[theorem]{Lemma}

\newtheorem{assumption}[theorem]{Assumption}
\newtheorem{remark}[theorem]{Remark}
\newtheorem{definition}[theorem]{Definition}

\newtheorem{main result}[theorem]{Main Result}

\definecolor{pastelred}{rgb}{1.0, 0.41, 0.38}
\definecolor{lightgreen}{rgb}{0.56, 0.93, 0.56}
\definecolor{lightblue}{rgb}{0.53, 0.81, 0.98}

\makeatletter
\newtheorem*{rep@theorem}{\rep@title}
\newcommand{\newreptheorem}[2]{%
\newenvironment{rep#1}[1]{%
 \def\rep@title{#2 \ref{##1}}%
 \begin{rep@theorem}}%
 {\end{rep@theorem}}}
\makeatother
\newreptheorem{theorem}{Theorem}
\newreptheorem{definition}{Definition}

\counterwithin{equation}{section}
\counterwithin{theorem}{section}

\newcommand{\mc}{\mathcal}

\renewcommand*\d{\mathop{}\!\mathrm{d}}

\newcommand{\R}{\mathbb{R}}
\newcommand{\N}{\mathbb{N}}

\newcommand{\T}{\ensuremath{\mathcal{T}}}
\newcommand{\set}[1]{\ensuremath{\left\{#1\right\}}}
\newcommand{\interior}[1]{\ensuremath{\mathring{#1}}}
\DeclareMathOperator*{\diam}{diam}
\newcommandx{\norm}[2][2=]{\ensuremath{\left\| #1 \right\|_{#2}}}
\newcommandx{\snorm}[2][2=]{\ensuremath{\left| #1 \right|_{#2}}}
\renewcommand{\P}{\ensuremath{\mathcal{P}}}
\newcommand{\Q}{\ensuremath{\mathcal{Q}}}
\newcommand{\poly}{\ensuremath{\mathbb{P}}}
\DeclareMathOperator*{\linspan}{span}
\newcommand{\V}{\ensuremath{\mathbb{V}}}
\newcommand{\proj}{\ensuremath{\Pi}}
\newcommand{\indicator}{\ensuremath{\mathbbm{1}}}
\newcommand{\inp}[2]{\mathopen{}\left\langle #1,\, #2\right\rangle\mathclose{}}
\DeclareMathOperator*{\gen}{gen}
\DeclareMathOperator*{\dist}{dist}
\newcommand{\QRG}{\textsf{Q-RG}}
\newcommand{\QR}{\textsf{Q-R}}
\newcommand{\QRB}{\textsf{Q-RB}}
\DeclareMathOperator*{\CE}{PE}

\title{$H^1$-Stability of the $L^2$-Projection onto Finite Element Spaces
\\on
Adaptively Refined Quadrilateral Meshes}

\author{Mazen Ali}
\address[MA]{Centrale Nantes, LMJL UMR CNRS 6629, France}
\email{mazen.ali@ec-nantes.fr}

\author{Stefan A.\ Funken \and Anja Schmidt}
\address[SF,AS]{Ulm University, Institute for Numerical Mathematics, Germany}
\email{\{stefan.funken,anja.schmidt\}@uni-ulm.de}

\date{\today}


\begin{document}

\begin{abstract}
    The $L^2$-orthogonal projection $\proj_h:L^2(\Omega)\rightarrow\V_h$
    onto a finite element (FE) space $\V_h$ is called $H^1$-stable
    iff
    $
        \norm{\nabla\proj_h u}[L^2(\Omega)]\leq C\norm{u}[H^1(\Omega)],
    $
    for any $u\in H^1(\Omega)$ with a positive constant
    $C\neq C(h)$ independent of the mesh size $h>0$.
    In this work, we discuss local criteria for the $H^1$-stability
    of adaptively refined meshes. We show that adaptive refinement
    strategies for quadrilateral meshes in 2D ({\QRG}
    and {\QRB}), introduced originally in
    Bank et al.\ 1982 and Kobbelt 1996, are $H^1$-stable for FE spaces of
    polynomial degree $p=2,\ldots,9$.
\end{abstract}

\maketitle

\section{Introduction}

Let $\Omega\subset\R^d$ be a bounded domain and
$\V_h\subset H^1(\Omega)$ a finite element (FE) space on $\Omega$.
The $H^1$-stability of the $L^2$-orthogonal projection
$\proj_h:L^2(\Omega)\rightarrow\V_h$ plays a key role in the
analysis of FE methods.
To mention a few examples:
in \cite{Tantardini},
the authors show that $H^1$-stability is equivalent to
the inf-sup stability and quasi-optimality of Galerkin methods
for parabolic equations;
$H^1$-stability is used in the analysis of multi-grid methods
(see \cite{yserentant1993}), boundary element methods (see
\cite{McLean1999,Steinbach2000,Steinbach1998}),
and adaptive methods (see \cite{Aurada}).
See also \cite{CC99, Wahlbin1995} for other applications.

For any given finite-dimensional $\V_h$, we trivially have
$\norm{\nabla\proj_h u}[L^2]\leq C\norm{u}[H^1]$ due to the equivalence
of norms on finite-dimensional spaces.
However, in general the constant $C=C(h)$ will depend on the
dimension of $\V_h$ or, equivalently, mesh size $h>0$.
The issue of $H^1$-stability is thus showing the constant
$C\neq C(h)$ does not depend on the dimension of $\V_h$.

It can be shown that for $H^1$-stability it is sufficient
to show stability in a weighted  $L^2$-norm
\begin{align}\label{eq:intro}
    \sum_{T\in\T}h_T^{-2}\norm{\proj_h u}[L^2(T)]^2
    \lesssim\sum_{T\in\T}h_T^{-2}\norm{u}[L^2(T)]^2,
\end{align}
see \Cref{sec:gencrit} for more details.
The above estimate is straight-forward if $\T$ is assumed to be
quasi-uniform, i.e.,
$(\max_Th_T)/(\min_Th_T)\sim 1$.
However, quasi-uniformity does not hold in general
for adaptively refined meshes.

\subsection{Previous Work}\label{sec:prwork}
Before we state the main contribution of this work, we briefly review the
key difficulties of showing $H^1$-stability and previous results
on this subject.

\subsubsection{Global Growth Condition}
In \cite{CT} (CT) the authors suggest criteria based on the localization
properties of the FE space $\V_h$ and the rate at which the element
size $h_T$ may vary. Namely, for any $u\in L^2(\Omega)$,
we can write\footnote{Assuming the meshing of $\Omega$ is exact.}
$u=\sum_{T\in\T}u_T$, with each $u_T$ supported only on $T$.
The localization property of $\proj_h$ from CT can be described by
a function $\gamma(T,\bar T)\geq 0$, decreasing with the distance between
$T$ and $\bar T$, such that
\begin{align*}
    \norm{\proj_hu_T}[L^2(\bar T)]\leq \gamma(T,\bar T)\norm{u_T}[L^2(\Omega)].
\end{align*}
Then, for any $T\in\T$ we can estimate
\begin{align*}
    h_T^{-1}\norm{\proj_h u}[L^2(T)]\leq
    h_T^{-1}\sum_{\bar T\in\T}\gamma(T,\bar T)\norm{u_{\bar T}}[L^2(T)]
    =\sum_{\bar T\in\T}\gamma(T,\bar T)
    \frac{h_{\bar T}}{h_T}h_{\bar T}^{-1}    
    \norm{u_{\bar T}}[L^2(T)].
\end{align*}
Summing over $T\in\T$, we can thus show \cref{eq:intro} if we can bound
\begin{align}\label{eq:growth}
    \sum_{\bar T\in\T}\gamma(T,\bar T)\frac{h_{\bar T}}{h_T}\lesssim 1,
\end{align}
independently of $T\in\T$.

The issue of $H^1$-stability hinges on the interplay between
the localization property of $\proj_h$ and the variation in size $h_T/h_{\bar T}$.
For common refinement strategies, e.g., newest vertex bisection (NVB)
or red green blue (RGB) refinement, the ratio $h_T/h_{\bar T}$ may grow exponentially, in the worst case;
while $\norm{\proj_h u_T}[L^2({\bar T})]$ will decay exponentially.
Whether \cref{eq:growth} is satisfied then depends on the factors
in the exponents of the growth/decay of both quantities.

E.g., \cref{eq:growth} is trivially satisfied for the non-conforming
case, see \Cref{thm:dg}. It is also satisfied for any projection for which
the support of $\proj_h u_T$ is finite, as is the case
for quasi-interpolation (Cl\'{e}ment-type) operators
(see also \Cref{ass:clement}).
Finally, we note that $H^1$-stability of the $L^2$-projection
is closely related to the question of the decay of the entries of
the inverse of
the mass-matrix away from the diagonal, see \cite{Demco}.

\subsubsection{Element-Wise Criteria}

The CT criterion \eqref{eq:growth} illustrates the main issues
of showing $H^1$-stability and all criteria proposed thereafter
are essentially based on the same idea. However,
\cref{eq:growth} is not easy to verify for common
adaptive refinement strategies.

In \cite{BPS} (BPS) the authors propose
a criteria that can be verified locally on an element $T\in\T$.
In \cite{CC} (CC) this was generalized to more flexible criteria
that can also be locally verified, and where
BPS and CT can be seen as specific instances of the CC criteria.

In \cite{BY} (BY) the authors suggest criteria that
can be verified by computing (small)
eigenvalue problems.
Though BY uses a proof technique different from CC,
it can be in fact seen as a particular instance of CC,
with the difference being that BY is easier to verify,
see also \Cref{rem:corr} for more details.

In \cite{Siebert} all of the above criteria are summarized into a single
framework that is both most flexible and
easiest to verify. Finally, in \cite{CCRGB,DirkNVB,GaspozNVB} the above criteria
were applied to show $H^1$-stability for adaptively
refined triangular meshes in 2D.

\subsection{This Work}

We condense the aforementioned criteria to a general
framework that, in principle, can be applied in any dimension
to meshes $\T$ consisting of arbitrary elements, with or without
hanging nodes, and of arbitrary polynomial order $p\in\N_0$.
We briefly show that for the non-conforming case
$\V_h\not\subset H^1(\Omega)$ stability is straight-forward and
requires no additional assumptions.
We specify criteria for the $H^1$-stability of regular and
$1$-irregular meshes in 2D, consisting of triangles,
general quadrilaterals, mixtures thereof, with $0$, $1$ or $2$
handing nodes per element, for various polynomial degrees $p\in\N$.
Our \textbf{main results} are
\Cref{thm:h1stab,thm:stabqrb}
where we show that the adaptive
refinement strategies for quadrilaterals {\QRG} and {\QRB}
from \cite{BankRG,Funken2020,Kobbelt} are $H^1$-stable
for polynomial degrees $p=2,\ldots,9$.

\subsection*{Outline}

In \Cref{sec:framework}, we discuss a general framework for verifying $H^1$-stability.
We show in \Cref{sec:nonconform} that the non-conforming case does not
require additional assumptions.
In \Cref{sec:comp}, we specify computable criteria and comment in more detail
on some practical aspects of verifying $H^1$-stability.
Our main results are contained in \Cref{sec:main}, where we state criteria
for $H^1$-stability of general triangular/quadrilateral meshes in 2D, recall
adaptive refinement strategies from \cite{Funken2020} and
prove $H^1$-stability for {\QRG} and {\QRB}.
In \Cref{app:tables}, we list tables of eigenvalues required
for verifying $H^1$-stability criteria. The corresponding code
can be found in \cite{codes}.

\subsection*{Notation}

We use $A\lesssim B$ for quantities $A,B\in\R$ to indicate
$A\leq CB$ for some constant $C\geq 0$ independent of $A$ or $B$.
Similarly $A\gtrsim B$ and $A\sim B$ if both $\lesssim$ and $\gtrsim$ hold.
We use the following shorthand notation
\begin{alignat*}{5}
    \norm{u}[0]&:=\norm{u}[L^2(\Omega)],\quad
    &&\norm{u}[1]&&:=\norm{u}[H^1(\Omega)],\quad
    &&\norm{u}[0, T]&&:=\norm{u}[L^2(T)],\\
    \norm{h^{-1}u}[0]^2&:=\sum_{T\in\T}h_T^{-2}\norm{u}[0,T]^2,\quad
    &&\inp{u}{v}_{0}&&:=\inp{u}{v}_{L^2(\Omega)},\quad
    &&\inp{u}{v}_{T}&&:=\inp{u}{v}_{L^2(T)}.
\end{alignat*}
Finally, we use $|T|$ to denote the Lebesgue $\R^d$-measure of $T$
and $\#\T$ to denote the standard counting measure of $\T$.
\section{A General Framework for $H^1$-Stability}\label{sec:framework}

There are several sufficient criteria for the $H^1$-stability
of the $L^2$-projection available in the literature, see
\cite{CC,CT,BPS,BY}.
These criteria were successfully applied to triangular meshes in 2D,
see \cite{CCRGB,DirkNVB,GaspozNVB}.
All of these criteria have a common underlying idea as explained
in \Cref{sec:prwork}.

In this section, we condense all of the aforementioned criteria to
a single unifying framework which can be
applied to either triangular or quadrilateral meshes, mixtures thereof, or
more general meshes.

\subsection{The Mesh and Finite Element Space}

Let $\Omega\subset\R^d$ be a bounded domain and
$\T:=\set{T_1,\ldots, T_N}$, $N\in\N$ a finite set of closed convex polytopes $T_i\subset\Omega$ which we refer to as
\emph{elements}. We make the following assumptions on $\T$.

\begin{definition}[Admissible Mesh]
    We call a mesh $\T$ \emph{admissible} if
    \begin{enumerate}[label=(\roman*)]
        \item    it is \emph{exact}, i.e., $\overline{\Omega}=\bigcup_{T\in\T}T$,
        \item    the elements $T\in\T$ are non-empty
                    $T\neq\emptyset$ and \emph{non-overlapping}, i.e.,
                    \begin{align*}
                        \interior{T_i}\cap\interior{T_j}=\emptyset,\quad T_i\neq T_j,    
                    \end{align*}
                    where $\interior{T}$ denotes the interior of $T$;           
        \item    the mesh $\T$ is \emph{shape-regular}. That is,
                    for any $T\in\T$, let $\rho_T$ denote the diameter of the largest
                    ball still contained in $T$ and let $h_T$ denote the diameter of
                    $T$
                    \begin{align*}
                        h_T := \diam(T) := \sup_{x,y\in T}\norm{x-y}[2].    
                    \end{align*}
                    Then, we assume $\rho_T\sim h_T$,
                    with a constant independent of $T$.
                    This also implies
                    $h_T^d\sim |T|$, where we use $|T|$ to denote the Lebesgue
                    $\R^d$-measure of $T$.                                 
    \end{enumerate}
\end{definition}

Let $H^1(\Omega)$ denote the Sobolev space of square-integrable functions
on $\Omega$
with one weak square-integrable derivative.
We define the sets of \emph{complete} and \emph{incomplete}
polynomials over an element $T$, respectively
\begin{align}\label{eq:setspoly}
    \P_p(T)&:=\linspan\set{x\mapsto x_1^{p_1}\cdots x_d^{p_d}:\;p_i\in\N_0,\;
    \sum_{i=1}^dp_i\leq p,\; x=(x_1,\ldots,x_d)^\intercal\in T},\\
    \Q_p(T)&:=\linspan\set{x\mapsto x_1^{p_1}\cdots x_d^{p_d}:\;p_i\in\N_0,\;
    p_i\leq p,\; x=(x_1,\ldots,x_d)^\intercal\in T}.\notag
\end{align}

\begin{definition}[Conforming FE Space]\label{def:FE}
    Let $\T$ be an admissible mesh and fix a polynomial degree $p\in\N$.
    A conforming piece-wise polynomial FE space of polynomial
    degree $p$ is any space $\V_h=\V_h(\T, p)$ such that
    \begin{align*}
        \V_h=\V_h(\T,p)=
        \set{v_h\in H^1(\Omega):(v_h)_{|T}\in\poly_p(T),\;T\in\T},
    \end{align*}
    where $\poly_p(T)\in\set{\P_p(T),\Q_p(T)}$, i.e.,
    $(v_h)_{|T}$ is either\footnote{Note that this permits
    some ambiguity as there can be (finitely many) spaces $\V_h(\T, p)$ that
    are conforming piece-wise polynomial FE spaces of degree $p$.
    The same is true since we, strictly speaking, allow for
    $\V_h\subset H^s(\Omega)$, $s>1$.
    However, this does not impact the theory presented here,
    as long as we \emph{fix} the choice of $\V_h$.
    In \Cref{sec:main}, this choice will be made clear.} a complete or incomplete
    polynomial of degree $p$.
    We denote the associated $L^2$-orthogonal projection by
    $\proj_h:L^2(\Omega)\rightarrow \V_h$.
    Here the subscript $h$ is used to indicate FE spaces and functions.
\end{definition}

We use $\mc V(\T)$ to denote the set of vertices
of all elements $T$, i.e., the set of extreme points. We associate to each
FE space $\V_h$ a finite set of (Lagrange) nodes $\mc N:=\mc N(\T)\subset\Omega$
and a set of basis functions
\begin{align*}
    \Phi_{\mc N}:=\set{\varphi_a:\;a\in\mc N(\T)},
\end{align*}
such that $\V_h=\linspan\Phi_{\mc N}$.
We set
$\mc V(T):=\mc V(\T)\cap T$.
Finally, for $S_a^T:=\set{x\in T:\;\varphi_a(x)\neq 0}$, we define
\begin{align}\label{eq:defNT}
    \mc N(T):=\set{a\in\mc N:\;|S_a^T|>0}.
\end{align}

\begin{remark}
    \leavevmode
    \begin{enumerate}[label=(\roman*)]
        \item    The conformity $\V_h\subset H^1(\Omega)$ is necessary
                    to ensure $\norm{\nabla\proj_hu}[0]$ is well-defined.
                    Since functions in $\V_h$ are polynomial except on a set of Lebesgue
                    measure zero, $v_h\in H^1(\Omega)$ holds if and only if
                    $v_h\in C(\Omega)$, or, equivalently, if $v_h$ is continuous
                    along the interior boundary
                    $\partial T\setminus\partial\Omega$ of $T\in\T$.
        \item    For ease of notation we do not consider boundary conditions
                    on $\partial\Omega$, but all of the subsequent results
                    apply to this case as well.
        \item    An important consequence of shape-regularity is that
                    for any $T\in\T$ and any $v_h\in\V_h$,
                    \begin{align*}
                        \norm{\nabla v_h}[0,T]\lesssim h_T^{-1}\norm{v_h}[0,T].
                    \end{align*}
    \end{enumerate}
\end{remark}

\subsection{General Criteria for $H^1$-Stability}\label{sec:gencrit}

The stability of $\proj_h$ in the $H^1$-norm
can be reduced to the stability of $\proj_h$ in a weighted $L^2$-norm
through the use of a stable interpolation operator.

\begin{assumption}[Stable Quasi-Interpolation Operator]\label{ass:clement}
    We assume the existence of a (possibly non-linear) mapping
    $Q_h:H^1(\Omega)\rightarrow\V_h$ that satisfies
    \begin{align}\label{eq:clement}
        \norm{\nabla Q_hu}[0]+\norm{h^{-1}(u-Q_hu)}[0]\lesssim\norm{u}[1],
        \quad\forall u\in H^1(\Omega),
    \end{align}
    where $h^{-1}:\Omega\rightarrow\R$ is the piece-wise constant
    function $h^{-1}:=\sum_{T\in\T}h_T^{-1}\indicator_T$,
    with the indicator functions $\indicator_T$.
\end{assumption}

An example of such a mapping $Q_h$ is the Cl\'{e}ment operator
\cite{Clement}
and variants thereof, see also \cite{SZ}.
Specifically for this work we can use the flexible construction
from \cite{CCClement}, which applies to both triangular and quadrilateral meshes,
with or without hanging nodes.

\begin{lemma}[Stability in a Weighted $L^2$-Norm]\label{lemma:weighted}
    Let $\T$ be an admissible mesh and let \Cref{ass:clement} hold.
    If $\proj_h:L^2(\Omega)\rightarrow\V_h$ satisfies
    \begin{align}\label{eq:weighted}
        \norm{h^{-1}\proj_h u}[0]\lesssim\norm{h^{-1}u}[0],
        \quad\forall u\in H^1(\Omega),
    \end{align}
    then $\norm{\nabla\proj_h u}[0]\lesssim\norm{u}[1]$ holds.
\end{lemma}

\begin{proof}
    This is a simple consequence of the triangle inequality,
    property \eqref{eq:clement} and shape-regularity
    \begin{align*}
        \norm{\nabla\proj_h u}[0]=\norm{\nabla\proj_h(u-Q_hu)
        +\nabla Q_h u}[0]\lesssim
        \norm{h^{-1}(u-Q_hu)}[0]+\norm{\nabla Q_hu}[0]\lesssim
        \norm{u}[1].
    \end{align*}
\end{proof}

We are thus left with showing \eqref{eq:weighted}.
To this end, we will use the following general criteria.

\begin{assumption}[$H^1$-stability Criteria]\label{ass:gencrit}
    We assume there exist (possibly non-linear) mappings $H_+$,
    $H_-:\V_h\rightarrow\V_h$ that satisfy
    \begin{enumerate}[label=(C\arabic*)]
        \item\label{C1}    the mapping $H_+$ is invertible and the inverse
                    $H_+^{-1}$ satisfies
                    \begin{align*}
                        \norm{h^{-1}v_h}[0]\lesssim
                        \norm{H_+^{-1}v_h}[0],\quad\forall v_h\in\V_h,
                    \end{align*}
        \item\label{C2}    the mapping $H_-$ satisfies
                    \begin{align*}
                        \norm{hH_-v_h}[0]\lesssim\norm{v_h}[0],\quad\forall
                        v_h\in\V_h,
                    \end{align*}
        \item\label{C3}    the mappings $H_+$ and $H_-$ jointly satisfy
                    \begin{align*}
                        \norm{v_h}[0]^2\lesssim
                        \inp{H_+v_h}{H_-v_h}_{0},                        
                        \quad\forall v_h\in\V_h.
                    \end{align*}
    \end{enumerate}
\end{assumption}

Different versions of the following theorem can be found in \cite{CC, BPS, Siebert}
for specific choices of $H_+$, $H_-$. We include the simple proof here
for completeness.

\begin{theorem}[$H^1$-stability \cite{CC, BPS, Siebert}]\label{thm:genh1stab}
    Let $\T$ be an admissible mesh and let \Cref{ass:clement}
    and \Cref{ass:gencrit} hold. Then, the $L^2$-orthogonal
    projection $\proj_h:L^2(\Omega)\rightarrow\V_h$ is $H^1$-stable
    in the sense
    \begin{align*}
        \norm{\nabla\proj_h u}[0]\lesssim\norm{u}[1],\quad\forall u\in H^1(\Omega).
    \end{align*}
\end{theorem}

\begin{proof}
    By \ref{C2}, \ref{C3} and $L^2$-orthogonality of $\proj_h$
    \begin{align*}
        \norm{H_+^{-1}\proj_h u}[0]^2&\lesssim
        \inp{\proj_hu}{H_-H_+\proj_hu}_0=
        \inp{h^{-1}u}{hH_-H_+\proj_hu}_0
        \leq \norm{h^{-1}u}[0]\norm{hH_-H_+^{-1}\proj_hu}[0]\\
        &\lesssim\norm{h^{-1}u}[0]\norm{H_+^{-1}\proj_hu}[0],
    \end{align*}
    and thus $\norm{H_+^{-1}\proj_hu}[0]\lesssim\norm{h^{-1}u}[0]$.
    From \ref{C1} we get
    $\norm{h^{-1}\proj_ju}[0]\lesssim\norm{H_+^{-1}\proj_hu}[0]$.
    Together with \Cref{lemma:weighted}, this completes the proof.
\end{proof}

\begin{remark}[Alternative Criteria]
    There is an alternative to criteria \ref{C1} -- \ref{C3}, which requires
    fewer assumptions. Namely, suppose there exists a linear operator
    $H_-:\V_h\rightarrow\V_h$ that satisfies
    \begin{enumerate}
        \item    for any $v_h\in\V_h$
                    \begin{align*}
                        \norm{h^{-1}v_h}[0]\lesssim\norm{H_-v_h}[0],
                    \end{align*}
        \item    and for any $v_h\in\V_h$
                    \begin{align*}
                        \norm{h(H_-)^*H_-v_h}[0]\lesssim\norm{h^{-1}v_h}[0],
                    \end{align*}
                    where $(H_-)^*$ denotes the Hilbert adjoint of $H_-$.
    \end{enumerate}
    Then, analogously to \Cref{thm:genh1stab}, we can show
    \begin{align*}
        \norm{h^{-1}\proj_hu}[0]^2&\lesssim
        \norm{H_-\proj_hu}[0]^2=
        \inp{(H_-)^*H_-\proj_hu}{\proj_h u}_0=
        \inp{(H_-)^*H_-\proj_hu}{u}\\
        &=
        \inp{h(H_-)^*H_-\proj_hu}{h^{-1}u}\leq
        \norm{h(H_-)^*H_-\proj_hu}[0]\norm{h^{-1}u}[0]
        \lesssim\norm{h^{-1}\proj_hu}[0]\norm{h^{-1}u}[0].
    \end{align*}
    
    The issue with this criteria is that, having specified $H_-$, it requires
    computing the adjoint $(H_-)^*$.
    In particular, even if $H_-$ has a simple ``local'' definition\footnote{In the sense
    that will become clear in \Cref{sec:comp}.}, $H_-^*$ will still be non-local in general.
    In contrast, criteria \ref{C1} -- \ref{C3} allow for the flexibility of choosing
    a local map $H_+$.
\end{remark}

\subsection{The Non-Conforming Case}\label{sec:nonconform}

Alternatively, we could consider $H^1$-stability of non-conforming
FE spaces in the broken Sobolev norm.
The gradient $\nabla:H^1(\Omega)\rightarrow L^2(\Omega)$ is replaced
by the piece-wise gradient $(\nabla_{\T}u)_{|T}:=\nabla(u_{|T})$ for
every $T\in\T$. The broken Sobolev space is defined as
\begin{align*}
    H^1(\T):=\set{u\in L^2(\Omega):\;u_{|T}\in H^1(T),\;
    T\in\T},
\end{align*}
and the corresponding FE space as
\begin{align*}
    \V_h^B:=\set{(v_h)_{|T}\in\poly_p(T),\;T\in\T}.
\end{align*}
Then, the corresponding $L^2$-projection $\proj_h^B:L^2(\Omega)
\rightarrow\V_h^B$ is said to be $H^1$-stable if
\begin{align*}
    \norm{\nabla_{\T}\proj_h^Bu}[0]\lesssim\norm{u}[H^1(\T)],
    \quad\forall u\in H^1(\T).
\end{align*}
However, in this case $H^1$-stability is trivially satisfied for any
admissible mesh that satisfies \Cref{ass:clement}.

\begin{theorem}[$H^1$-stability Discontinuous FE]\label{thm:dg}
    Let $\T$ be an admissible mesh satisfying \Cref{ass:clement} (with
    the definition of norms adjusted accordingly). Then,
    the $L^2$-projection $\proj_h^B:L^2(\Omega)\rightarrow\V_h^B$
    is $H^1$-stable.
\end{theorem}

\begin{proof}
    Let $u_T\in L^2(\Omega)$ be an arbitrary $L^2$-function supported
    on $T\in\T$. Let $w_T$ be the $L^2(T)$-orthogonal projection
    of $u_T$ onto $\poly_p(T)$, which we extend with zero on
    $\Omega\setminus T$. Then, clearly $w_T\in\V_h^B$ and, on one hand,
    $\norm{u_T-w_T}[0]\geq \norm{u_T-\proj_h^Bu_T}[0]$.
    On the other hand,
    \begin{align*}
        \norm{u_T-w_T}[0]^2=\sum_{\bar T\in\T}\norm{u_T-w_T}[0,\bar T]^2
        \leq \sum_{\bar T\in\T}\norm{u_T-\proj_h^Bu_T}[0,\bar T]^2
        =\norm{u_T-\proj_h^Bu_T}[0]^2,
    \end{align*}
    and thus $\norm{u_T-w_T}[0]=\norm{u_T-\proj_h^Bu_T}[0]$.
    Since $\V_h^B\subset L^2(\Omega)$ is compact, the orthogonal projection
    is unique and thus $w_T=\proj_h^Bu_T$.
    
    In particular, this implies
    \begin{align*}
        \norm{\proj_h^Bu_T}[0,\bar T]=0,\quad\text{for any }
        \bar T\neq T.
    \end{align*}
    Since $\proj_h$ is an orthogonal projection, we also have
    $\norm{\proj_h^Bu_T}[0]\leq\norm{u_T}[0]$.
    
    For any $u\in H^1(\T)$, we can write $u=\sum_{T\in\T}u_T$ with
    $u_T:=u\indicator_T$ and observe
    $\proj_h^Bu=\sum_{T\in\T}\proj_h^Bu_T$.
    Thus $\norm{\proj_h^Bu}[0,T]=\norm{\proj_h^Bu_T}[0]$ and
    \begin{align*}
        \norm{h^{-1}\proj_h^Bu}[0]^2=\sum_{T\in\T}
        h_T^{-2}\norm{\proj_h^Bu}[0,T]^2=
        \sum_{T\in\T}h_T^{-2}\norm{\proj_h^Bu_T}[0]^2
        \leq\sum_{T\in\T}h_T^{-2}\norm{u_T}[0]^2=
        \norm{h^{-1}u}[0]^2.    
    \end{align*}
    Together with \Cref{lemma:weighted} (adjusted for the broken Sobolev norm),
    this shows the $H^1$-stability of $\proj_h^B$.     
\end{proof}
\section{Computable Criteria for $H^1$-Stability}\label{sec:comp}

In this section, we discuss a particular choice for the mappings
$H_+$ and $H_-$. First, we briefly motivate how a
``practical'' choice for $H_+$, $H_-$ would look like.

Let $\{\T_n\}_{n=0}^\infty$ be a sequence of finer meshes
with $\#\T_n\rightarrow\infty$ and let $\proj_{h_n}$ denote
the corresponding $L^2$-projections.
For this particular sequence, $H^1$-stability means we have
\begin{align*}
    \norm{\nabla\proj_{h_n}u}[0]\leq C\norm{u}[1],\quad\forall u\in H^1(\Omega),
\end{align*}
for some $C\neq C(n)$ independent of $n\in\N$.
The sequence $\{\T_n\}_{n=0}^\infty$ may depend, among other things,
on the initial discretization, the problem to be solved (such as a partial differential
equation), the choice of marking strategy (e.g., error estimator),
the choice of adaptive refinement strategy and so on.

For this reason disproving $H^1$-stability can be particularly difficult:
even if $C=C(n)\rightarrow\infty$ for some artificially constructed sequence
$\{\T_n\}_{n=0}^\infty$, we can still have a class of problems and a set
of marking rules for which $\{\T_n\}_{n=0}^\infty$ will always remain
$H^1$-stable.
Thus, proofs of $H^1$-stability as in \cite{CCRGB, GaspozNVB, DirkNVB}
focus solely on the
refinement strategy (e.g., RGB, NVB, RG, etc.). This
necessarily results in the (much stronger)
\emph{local} conditions, i.e.,
the conditions of \Cref{ass:gencrit} are replaced with
\begin{align}\label{eq:localC}
    \norm{h^{-1}v_h}[0,T]\lesssim \norm{H_+^{-1}v_h}[0,T],\quad
    \norm{hH_-v_h}[0,T]\lesssim\norm{v_h}[0,T],\quad
    \norm{v_h}[0,T]^2\lesssim\inp{H_+v_h}{H_-v_h}_{0,T},
\end{align}
for any $T\in\T$ and any $v_h\in\V_h$.

\subsection{Verifying \ref{C1} -- \ref{C2}}

In \cite{CC, Siebert, BPS} the authors consider locally defined weight functions as follows.

\begin{lemma}[Choice of $H_+$, $H_-$ \cite{CC, Siebert, BPS}]\label{lemma:choiceh}
    Let $\set{h_a^+:\;a\in\mc N}$ and $\set{h_a^-:\;a\in\mc N}$ be sets
    of positive weights $h_a^+,h_a^->0$.
    We define the mapping $H_+:\V_h\rightarrow\V_h$ as
    \begin{align*}
        H_+(v_h)=H_+\left(\sum_{a\in\mc N}c_a\varphi_a\right):=
        \sum_{a\in\mc N}h_a^+c_a\varphi_a,    
    \end{align*}
    and analogously $H_-:\V_h\rightarrow\V_h$.
    If the weights satisfy for any $T\in\T$
    \begin{align*}
        h_a^+\sim h_T,\quad h_a^{-1}\sim h_T^{-1},\quad
        \forall a\in\mc N(T),    
    \end{align*}
    then $H_+$ and $H_-$ satisfy \ref{C1} -- \ref{C2}.         
\end{lemma}

\begin{proof}
    By the definition of the index set $\mc N(T)$ (see
    \eqref{eq:defNT}), we have for any $T\in\T$ and any
    $v_h=\sum_{a\in\mc N}c_a\varphi_a\in\V_h$
    \begin{align*}
        \norm{h^{-1}v_h}[0,T]=
        \norm{\sum_{a\in\mc N(T)}h_T^{-1}c_a\varphi_a}[0,T]
        \lesssim
        \norm{\sum_{a\in\mc N(T)}h_a^{-1}c_a\varphi_a}[0,T]
        =\norm{H_+^{-1}v_h}[0,T],
    \end{align*}
    and consequently $\norm{h^{-1}v_h}[0]\lesssim
    \norm{H_+^{-1}v_h}[0]$. Analogously for \ref{C2}.
\end{proof}

\subsection{Verifying \ref{C3}}

The last condition can be verified by solving a local generalized
eigenvalue problem. We namely have
\begin{align*}
    \norm{v_h}[0, T]^2\lesssim
    \inp{H_+v_h}{H_-v_h}_{0,T}
    =\inp{v_h}{(H_+)^*H_-v_h}_{0,T}
    =\inp{v_h}{\frac{1}{2}(H_+^*H_-+H_-^*H_+)v_h}_{0,T}.
\end{align*}
This amounts to assembling the local mass-matrix $M(T)$
\begin{align*}
    (M(T))_{a,b}:=\inp{\varphi_a}{\varphi_b}_{0,T}\quad
    a,b\in\mc N(T),
\end{align*}
the weighted matrix $A(T)$
\begin{align*}
    (A(T))_{a,b}:=\frac{1}{2}\left(h_a^+h_b^-+
    h_b^+h_a^-\right)(M(T))_{a,b},\quad a,b\in\mc N(T),
\end{align*}
and solving the generalized eigenvalue problem
\begin{align}\label{eq:evalue}
    A(T)x=\lambda M(T)x,\quad x\in\R^{\#\mc N(T)}.
\end{align}
If the smallest eigenvalue of \eqref{eq:evalue} $\lambda_{\min}>0$ is positive,
then \ref{C3} holds for a positive constant $C:=\lambda_{\min}^{-1}$.
Moreover, if $T$ can be obtained by an affine transformation
$B_T:\hat T\rightarrow T$ from some reference element $\hat T$,
then the minimal eigenvalue in \eqref{eq:evalue} does not
depend on $T$.
We comment more on this in \Cref{rem:calc}
and we address the case of non-linear transformations
$B_T:\hat T\rightarrow T$ in \Cref{sec:nonlinear}.

\subsection{Weights Based on the Refinement Level}

In \cite{Siebert, CC, GaspozNVB} the authors consider weights based on the refinement level
of an element and a distance function.
We will use the same type of function for this work.

\begin{definition}[Generation of an Element]
    For a sequence $\{\T_n\}_{n=0}^\infty$,
    we assume\footnote{This is the case for most common refinement
    strategies and also holds for the strategy considered in \Cref{sec:review}.}
    that any $T\in\T_n$, for any $n\in\N_0$, has a macro element
    $K_T\in\T_0$ such that $T\subset K_T$. The generation $\gen(T)\geq 0$
    of an element $T\in\T_n$ is defined as
    \begin{align*}
        \gen(T):=\log_2\left(\frac{|K_T|}{|T|}\right).
    \end{align*}
\end{definition}

\begin{definition}[Distance Function]
    For two nodes $a,b\in\mc N(\T)$, $a\neq b$, the distance
    function $\dist(a,b)\in\N_0$ is defined as the minimal $J\in\N$
    such that there exists elements $T_1,\ldots,T_J$ with
    \begin{align*}
        a\in T_1, T_1\cap T_2\neq\emptyset,\ldots,
        T_{J-1}\cap T_J\neq\emptyset,T_J\ni b.
    \end{align*}
    For $a=b$, set $\dist(a,a)=0$.
    The distance to an element is defined as
    \begin{align*}
        \dist(a, T):=\min\set{\dist(a,b):\;b\in\mc N(T)}.
    \end{align*}
\end{definition}

With this we can finally define

\begin{definition}[Weight Function \cite{Siebert}]\label{def:siebert}
    For any $z\in\mc V(\T)$ and a fixed parameter $\mu>0$, define the weight
    \begin{align*}
        h_z:=\min\set{2^{(\mu\dist(z,T)-\gen(T))/d}:\;T\in\T}.
    \end{align*}
    Then, for any $a\in\mc N(\T)\cap\mc V(\T)$,
    $h_a^+$ is defined as above, while
    $h_a^-:=(h_a)^{-1}$.
    For any $a\in\mc N(\T)\setminus\mc V(\T)$,
    $h_a^+$ is
    defined\footnote{This choice is slightly arbitrary
    and is made for convenience, as any choice
    satisfying \ref{C1} -- \ref{C3} would be valid.}
    by a linear interpolation of $h_z^+$,
    $z\in\mc V(T)\cap\mc N(T)$. Analogously
    for $h_a^-$, $a\in\mc N(\T)\setminus\mc V(\T)$.
\end{definition}

\begin{remark}[Correspondence to Other Criteria]\label{rem:corr}
    The framework introduced in \Cref{sec:framework,sec:comp}
    covers the $H^1$-stability criteria used in
    \cite{BPS, CT, CC, Siebert, BY}.
    In particular, in \cite{CC} it was shown how the CT and BPS
    criteria correspond to a specific choice of weights $h_z$.
    The iterative approach introduced in \cite{BY} 
    corresponds to $\mu=2$ and, once again, a specific choice of $h_z$.
    Finally, in \cite{Siebert} the authors consider $h_z$ as stated in \Cref{def:siebert}.
\end{remark}

To ensure that this choice of $H_+$, $H_-$ satisfies \ref{C1} -- \ref{C3},
we first check the conditions of \Cref{lemma:choiceh}, i.e.,
if $h_z\sim h_T$ for any $z\in\mc V(T)$ and any $T\in\T$.
Note that one inequality is easily satisfied: for a
sequence of admissible meshes $\{\T_n\}_{n=0}^\infty$,
$z\in\mc V(T^*)$ and a macro element $K_{T^*}\supset T^*$,
$K_{T^*}\in\T_0$, we have
\begin{align*}
    h_z&=\min\set{2^{(\mu\dist(z,T)-\gen(T))/d}:\;T\in\T_n}
    \leq 2^{(\mu\dist(z,T^*)-\gen(T^*))/d}
    =2^{-\gen(T^*)/d}\\
    &=\left(|T^*|/|K_{T^*}|\right)^{1/d}\lesssim
    h_{T^*}\left(1/|K_{T^*}|\right)^{1/d},
\end{align*}
where we used the definition of $\gen(\cdot)$ and the
shape-regularity of $\T_n$.
Thus, $h_z\lesssim h_T$ is easily satisfied with a constant depending
only on $\T_0$.

On the other hand, we have the critical inequality $h_T\lesssim h_z$.
In principle this can always be satisfied by choosing $\mu>0$
large enough. However, a larger $\mu>0$ implies we allow
for a larger rate of change for sizes of neighboring elements.
We namely have

\begin{lemma}[\cite{CC,GaspozNVB,Siebert}]
    For any $T\in\T$, it holds
    \begin{align}\label{eq:hzratio}
        \max_{z,z'\in\mc V(T)}\frac{h_z}{h_{z'}}\leq 2^{\mu/d}.
    \end{align}
\end{lemma}

The parameter $\mu>0$ is determined by the refinement
strategy and reflects the fact that neighboring
elements $T,\bar T\in\T$ can vary in size upto
\begin{align}\label{eq:mu}
    \max\set{|T|/|\bar T|,|\bar T|/|T|}\lesssim 2^\mu.
\end{align}
We are thus interested in the smallest $\mu$,
for a given refinement strategy, satisfying
\eqref{eq:mu}, or, equivalently, satisfying $h_T\lesssim h_z$.

\begin{remark}[Necessary Conditions]\label{rem:necessary}
    The local conditions \ref{C1} -- \ref{C3} from \eqref{eq:evalue}
    are sufficient for $H^1$-stability but
    are by no means necessary. The proof implicitly assumes
    that $h_T/h_{\bar T}$ increases exponentially with
    the distance between elements $T$ and $\bar T$.
    Condition \ref{C3} then ensures this exponential
    growth is counterbalanced by the exponential decay
    of $\norm{\proj_h u_T}[0,\bar T]$, see also \cref{eq:growth},
    the discussion thereafter
    and \Cref{thm:dg}.
    However, assuming consistent worst-case exponential growth for
    $h_T/h_{\bar T}$ is overly pessimistic and will not
    occur for many reasonable adaptively refined meshes
    $\T$, see also \cite[Section 6]{BY}.
\end{remark}
\section{$H^1$-Stability of Adaptively Refined Quadrilateral Meshes}
\label{sec:main}

As discussed in \Cref{sec:comp}, to show $H^1$-stability of a general mesh using
weights as defined in \Cref{def:siebert}, we have to determine the smallest
$\mu>0$ for a given refinement strategy such that $h_T\lesssim h_z$,
$z\in\mc V(T)$ and check whether the smallest eigenvalue in
\cref{eq:evalue} is positive. In this section, we proceed in two
steps.

First, we vary $\mu=1,2,3,4$ and calculate for each $\mu$ and
several reference elements the smallest eigenvalue
in \cref{eq:evalue}.
For a polynomial degree $p\in\N$,
the corresponding conforming
continuous Lagrange FE space $\V_h$
is defined as the space of piece-wise polynomial
functions such that (see also \eqref{eq:setspoly} and \Cref{def:FE})
\begin{itemize}
    \item $(v_h)_{|T}\in\P_p(T)$ if $T$ is a triangle,
    \item $(v_h)_{|T}\in\Q_p(T)$ if $T$ is a quadrilateral,
    \item $v_h\in H^1(\Omega)$ and $v_h\not\in H^2(\Omega)$.
\end{itemize}
We list these results in \Cref{app:tables}.

Second, we use the results from \Cref{app:tables} together with a proof for the minimal
$\mu>0$ to show $H^1$-stability for the refinement strategy
{\QRG} from \cite{BankRG,Funken2020} (see also \Cref{sec:review}).
We will discuss the case of general quadrilaterals and
non-linear transformations $B_T:\hat T\rightarrow T$
in \Cref{sec:nonlinear}, and conclude by showing
$H^1$-stability for {\QRB}.

\begin{remark}[Scope of Results]
    The results in \Cref{app:tables} are both used in this
    work to prove $H^1$-stability
    for particular refinement strategies, and they are
    intended as a reference to check $H^1$-stability
    for other refinement strategies not considered here.
    The corresponding code can be found in \cite{codes}.
\end{remark}

\begin{remark}[Calculating Eigenvalues in \Cref{eq:evalue}]\label{rem:calc}
    To calculate eigenvalues in \cref{eq:evalue}, we make use
    of a handy observation from \cite{Siebert}. For $h_z$ as defined in \Cref{def:siebert},
    for any $z\in \mc V(T^*)$, we have
    \begin{align}\label{eq:hzmin}
        h_z^{-2}=\max\set{2^{\gen(T)-\mu\dist(z,T)}:\;T\in\T}\geq\gen(T^*)\geq 0.
    \end{align}        
    Moreover, the term $\gen(T)-\mu\dist(z, T)$ --
    due to the definition of $\dist(z, T)$ and $\gen(T)$ --
    can only attain values in a discrete set. E.g.,
    for the refinement strategy {\QRG} from \cite{BankRG,Funken2020} (see also \Cref{sec:review}),
    $\gen(T)\in\set{0, 1, \log_2(8/3),2,3,2+\log_2(8/3),4,5,4+\log_2(8/3),\ldots}$.
    Thus, normalizing by the largest (or smallest) $h_z$ in an element
    $T\in\T$, $z\in\mc V(T)$, together with
    \cref{eq:hzmin} and \cref{eq:hzratio},
    for any $T\in\T$ there is only a finite number of
    possible weight configurations. Consequently, in
    \Cref{app:tables} we computed the smallest eigenvalue for each
    combination and took the minimum over all combinations.
\end{remark}

\subsection{Adaptive Refinement for Quadrilateral Meshes}\label{sec:review}

In this subsection, we introduce some refinement strategies for
quadrilateral meshes from \cite{BankRG,Funken2020,verfuerth,Kobbelt}. We consider three different strategies: red refinement {\QR}, red-green refinement {\QRG} and red-blue refinement {\QRB} on quadrilaterals. Red, green and blue patterns used within these refinement strategies are depicted in Figure~\ref{fig:refinementstrategy}. Although these refinement
strategies can be in principle applied to initial meshes of general quadrilaterals, we only consider initial meshes consisting of parallelograms in this work.

 \begin{figure}[h!]
 \begin{minipage}{0.15\textwidth}
\begin{tikzpicture}
\coordinate (1) at (-1,-1);
\coordinate (2) at (1,-1);
\coordinate (3) at (1,1);
\coordinate (4) at (-1,1);
\draw[fill=pastelred] (1)--(2)--(3)--(4)--(1);
\fill (1) circle (2pt);
\fill (2) circle (2pt);
\fill (3) circle (2pt);
\fill (4) circle (2pt);
\node at (0,-1.5) {none};
\end{tikzpicture}
\end{minipage}
\begin{minipage}{0.15\textwidth}
\begin{tikzpicture}
\coordinate (1) at (-1,-1);
\coordinate (2) at (1,-1);
\coordinate (3) at (1,1);
\coordinate (4) at (-1,1);
\coordinate (5) at (0,-1);
\coordinate (6) at (1,0);
\coordinate (7) at (0,1);
\coordinate (8) at (-1,0);
\draw[fill=pastelred] (1)--(2)--(3)--(4)--(1);
\draw (5)--(7);
\draw (6)--(8);
\fill (1) circle (2pt);
\fill (2) circle (2pt);
\fill (3) circle (2pt);
\fill (4) circle (2pt);
\fill (5) circle (2pt);
\fill (6) circle (2pt);
\fill (7) circle (2pt);
\fill (8) circle (2pt);
\fill ($(5)!0.5!(7)$) circle (2pt);
\node at (0,-1.5) {red};
\end{tikzpicture}
\end{minipage}
\begin{minipage}{0.15\textwidth}
\begin{tikzpicture}
\coordinate (1) at (-1,-1);
\coordinate (2) at (1,-1);
\coordinate (3) at (1,1);
\coordinate (4) at (-1,1);
\draw[fill=lightgreen] (1)--(2)--(3)--(4)--(1);
\draw (4)--($(1)!0.5!(2)$)--(3);
\fill (1) circle (2pt);
\fill (2) circle (2pt);
\fill (3) circle (2pt);
\fill (4) circle (2pt);
\fill ($(1)!0.5!(2)$) circle (2pt);
\node at (0,-1.5) {green 1};
\end{tikzpicture}
\end{minipage}
\begin{minipage}{0.15\textwidth}
\begin{tikzpicture}
\coordinate (1) at (-1,-1);
\coordinate (2) at (1,-1);
\coordinate (3) at (1,1);
\coordinate (4) at (-1,1);
\draw[fill=lightgreen] (1)--(2)--(3)--(4)--(1);
\draw (4)--($(1)!0.5!(2)$)--($(2)!0.5!(3)$)--(4);
\fill (1) circle (2pt);
\fill (2) circle (2pt);
\fill (3) circle (2pt);
\fill (4) circle (2pt);
\fill ($(1)!0.5!(2)$) circle (2pt);
\fill ($(2)!0.5!(3)$) circle (2pt);
\node at (0,-1.5) {green 2};
\end{tikzpicture}
\end{minipage}
\begin{minipage}{0.15\textwidth}
\begin{tikzpicture}
\coordinate (1) at (-1,-1);
\coordinate (2) at (1,-1);
\coordinate (3) at (1,1);
\coordinate (4) at (-1,1);
\draw[fill=lightgreen] (1)--(2)--(3)--(4)--(1);
\draw ($(1)!0.5!(2)$)--($(3)!0.5!(4)$);
\fill (1) circle (2pt);
\fill (2) circle (2pt);
\fill (3) circle (2pt);
\fill (4) circle (2pt);
\fill ($(1)!0.5!(2)$) circle (2pt);
\fill ($(3)!0.5!(4)$) circle (2pt);
\node at (0,-1.5) {green 3};
\end{tikzpicture}
\end{minipage}
\begin{minipage}{0.15\textwidth}
\begin{tikzpicture}
\coordinate (1) at (-1,-1);
\coordinate (2) at (1,-1);
\coordinate (3) at (1,1);
\coordinate (4) at (-1,1);
\coordinate (5) at (0,-1);
\coordinate (6) at (1,0);
\coordinate (7) at (0,1);
\coordinate (8) at (-1,0);
\coordinate (9) at (0,0);
\draw[fill=lightblue] (1)--(2)--(3)--(4)--(1);
\draw (5)--(9);
\draw (6)--(9);
\draw (9)--(4);
\fill (1) circle (2pt);
\fill (2) circle (2pt);
\fill (3) circle (2pt);
\fill (4) circle (2pt);
\fill (5) circle (2pt);
\fill (6) circle (2pt);
\fill (9) circle (2pt);
\node at (0,-1.5) {blue};
\end{tikzpicture}
\end{minipage}
 \caption{From left to right: Unrefined element, red pattern, three different green patterns and a blue pattern.}
 \label{fig:refinementstrategy}
 \end{figure}
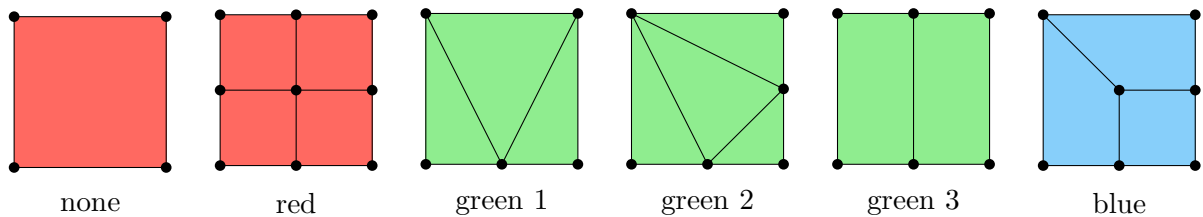

We call a quadrilateral \emph{red}-refined if it is subdivided into four quadrilaterals by joining the midpoints of opposite edges with each other. We call a quadrilateral \emph{green}-refined if it is divided into
\begin{enumerate}
\item three subtriangles by connecting the midpoint of one edge with the vertices opposite to this edge (green 1);
\item four subtriangles by joining the midpoints of two adjacent edges with the vertex shared by the other two edges with each other (green 2);
\item two subquadrilaterals by joining the midpoints of opposite edges with each other (green 3).
\end{enumerate}
We call a quadrilateral \emph{blue}-refined if it is divided into three subquadrilaterals by connecting the quadrilateral's midpoint with the two midpoints of adjacent edges and with the vertex shared by the other two edges with each other.

With this we can describe the refinement strategies {\QR} in Algorithm~\ref{alg:QR}, {\QRG} in Algorithm~\ref{alg:QRG} and {\QRB} in Algorithm~\ref{alg:QRB}.
Note, that for {\QR} 1-irregularity of the mesh is ensured. We call a mesh $\mathcal{T}$ \emph{1-irregular} if the number of hanging nodes per edge is restricted to one. A node $z \in \mc V(\T)$ is called a \emph{hanging node} if for some element $T \in \mathcal{T}$ holds $z \in \partial T\setminus\mc V(\T)$.

\begin{algorithm}[h!]
\caption{{\QR}}
\label{alg:QR}
\begin{algorithmic}[1]
\State{\bf Input:} Mesh $\mathcal{T}$ and set of marked elements $\mathcal{T}_M$.
\State {\bf Output:} Refined mesh $\hat{\mathcal{T}}$.
\Repeat
\State red-refine all $T \in \mathcal{T}_M$
\State add elements with more than one hanging node per edge to $\mathcal{T}_M$ \Comment{1-Irregularity}
\State add elements with more than three refined neighbors to $\mathcal{T}_M$ \Comment{3-Neighbor Rule}
\Until{$\mathcal{T}_M = \emptyset$}
\end{algorithmic}
\end{algorithm}

\begin{remark}[Reference Elements for 1-Irregular Meshes]
For 1-irregular meshes, three different situations have to be considered: 0, 1 or 2 hanging nodes per element, see Figure~\ref{fig:hangingnode}. Other situations can not arise because a red refinement does not allow for any other combinations of hanging nodes.
\begin{figure}[h!]
 \begin{minipage}{0.25\textwidth}
\begin{tikzpicture}
\coordinate (1) at (-1,-1);
\coordinate (2) at (1,-1);
\coordinate (3) at (1,1);
\coordinate (4) at (-1,1);
\coordinate (5) at (-3,1);
\coordinate (6) at (-3,-1);
\coordinate (7) at (-1,-3);
\coordinate (8) at (1,-3);
\draw (1)--(2)--(3)--(4)--(1);
\draw (4)--(5)--(6)--(1);
\draw (1)--(7)--(8)--(2);
\draw ($(5)!0.5!(6)$) -- ($(2)!0.5!(3)$);
\draw ($(4)!0.5!(3)$) -- ($(7)!0.5!(8)$);
\draw ($(5)!0.5!(4)$) -- ($(1)!0.5!(6)$);
\draw ($(1)!0.5!(7)$) -- ($(2)!0.5!(8)$);
\fill (1) circle (2pt);
\fill ($(1)!0.5!(2)$) circle (2pt);
\fill ($(1)!0.5!(3)$) circle (2pt);
\fill ($(1)!0.5!(4)$) circle (2pt);
\draw[pattern=crosshatch dots] (1)--($(1)!0.5!(2)$)--($(1)!0.5!(3)$)--($(1)!0.5!(4)$);
\end{tikzpicture}
\end{minipage} \hspace*{3ex}
 \begin{minipage}{0.25\textwidth}
\begin{tikzpicture}
\coordinate (1) at (-1,-1);
\coordinate (2) at (1,-1);
\coordinate (3) at (1,1);
\coordinate (4) at (-1,1);
\coordinate (5) at (-3,1);
\coordinate (6) at (-3,-1);
\coordinate (7) at (-1,-3);
\coordinate (8) at (1,-3);
\draw (1)--(2)--(3)--(4)--(1);
\draw (4)--(5)--(6)--(1);
\draw (1)--(7)--(8)--(2);
\draw ($(1)!0.5!(4)$) -- ($(2)!0.5!(3)$);
\draw ($(4)!0.5!(3)$) -- ($(7)!0.5!(8)$);
\draw ($(1)!0.5!(7)$) -- ($(2)!0.5!(8)$);
\fill (1) circle (2pt);
\fill ($(1)!0.5!(2)$) circle (2pt);
\fill ($(1)!0.5!(3)$) circle (2pt);
\draw[pattern=crosshatch dots] (1)--($(1)!0.5!(2)$)--($(1)!0.5!(3)$)--($(1)!0.5!(4)$);
\node[circle,draw=black, fill=white,inner sep=0pt,minimum size=4pt] at ($(1)!0.5!(4)$) {};
\end{tikzpicture}
\end{minipage} \hspace*{3ex}
 \begin{minipage}{0.25\textwidth}
\begin{tikzpicture}
\coordinate (1) at (-1,-1);
\coordinate (2) at (1,-1);
\coordinate (3) at (1,1);
\coordinate (4) at (-1,1);
\coordinate (5) at (-3,1);
\coordinate (6) at (-3,-1);
\coordinate (7) at (-1,-3);
\coordinate (8) at (1,-3);
\draw (1)--(2)--(3)--(4)--(1);
\draw (4)--(5)--(6)--(1);
\draw (1)--(7)--(8)--(2);
\draw[pattern=crosshatch dots] (1)--($(1)!0.5!(2)$)--($(1)!0.5!(3)$)--($(1)!0.5!(4)$);
\draw ($(1)!0.5!(4)$) -- ($(2)!0.5!(3)$);
\draw ($(4)!0.5!(3)$) -- ($(1)!0.5!(2)$);
\fill (1) circle (2pt);
\fill ($(1)!0.5!(3)$) circle (2pt);
\node[circle,draw=black, fill=white,inner sep=0pt,minimum size=4pt] at ($(1)!0.5!(4)$) {};
\node[circle,draw=black, fill=white,inner sep=0pt,minimum size=4pt] at ($(1)!0.5!(2)$) {};
\end{tikzpicture}
\end{minipage}
\caption{From left to right: 0, 1 and 2 hanging nodes (in white) per element (dotted).}
\label{fig:hangingnode}
\end{figure}
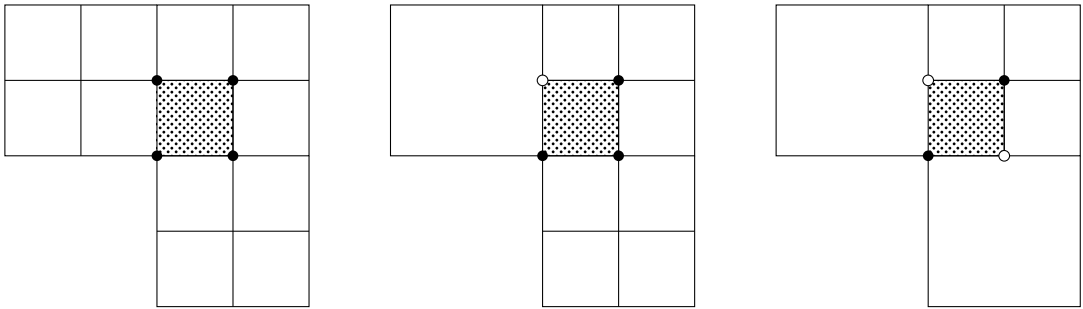
\end{remark}

\begin{algorithm}[h!]
\caption{{\QRG}}
\label{alg:QRG}
\begin{algorithmic}[1]
\State{\bf Input:} Mesh $\mathcal{T}$ and set of marked elements $\mathcal{T}_M$.
\State {\bf Output:} Refined mesh $\hat{\mathcal{T}}$.
\State Undo all green refinements and add their parent element to $\mathcal{T}_M$.
\State Call {\QR} with updated mesh and updated marked elements $\mathcal{T}_M$.
\State Eliminate hanging nodes by matching green patterns.
\end{algorithmic}
\end{algorithm}

\begin{algorithm}[h!]
\caption{{\QRB}}
\label{alg:QRB}
\begin{algorithmic}[1]
\State{\bf Input:} Mesh $\mathcal{T}$ and set of marked elements $\mathcal{T}_M$.
\State {\bf Output:} Refined mesh $\hat{\mathcal{T}}$.
\State Undo all blue refinements and add their parent element to $\mathcal{T}_M$.
\State Call {\QR} with updated mesh and updated marked elements $\mathcal{T}_M$.
\State Eliminate hanging nodes by matching blue patterns and an additional \textsc{Closure} step, see \cite{Funken2020}.
\end{algorithmic}
\end{algorithm}

\begin{remark}[Generation of an Element]
Assume the initial mesh $\T_0$ consists of parallelograms.
A red refinement quarters a parallelogram, i.\,e.,
the area ratios and thus the possible generations are given by $$\gen(T) \in \left\{0,2,4,6,\ldots\right\},$$ for all $T \in \mathcal{T}$.

For {\QRG} we have
$$\gen(T) \in \left\{0, 1, \log_2(8/3),2,3,2+\log_2(8/3),4,5,4+\log_2(8/3),\ldots\right\},$$ for all $T \in \mathcal{T}$,
cf. \Cref{fig:ratios}. As green refinements are undone before they are further refined, the sequence continues in the same scheme.

As shown in \Cref{fig:refinementstrategy}, not all elements of the blue patterns can be obtained by an affine transformation $B_T:\hat T\rightarrow T$ from $\hat T=[0,1]^2$, see \Cref{sec:nonlinear}.
We have for {\QRB} that $$\gen(T) \in \left\{0,\log_2(8/3),2,2+\log_2(8/3),4,\ldots\right\},$$ for all $T \in \mathcal{T}$.
\end{remark}

\begin{figure}[h!]
 \begin{minipage}{0.15\textwidth}
\begin{tikzpicture}
\coordinate (1) at (-1,-1);
\coordinate (2) at (1,-1);
\coordinate (3) at (1,1);
\coordinate (4) at (-1,1);
\draw(1)--(2)--(3)--(4)--(1);
\fill (1) circle (2pt);
\fill (2) circle (2pt);
\fill (3) circle (2pt);
\fill (4) circle (2pt);
\node at (0,0) {$T_0$};
\node at (0,-1.5) {none};
\end{tikzpicture}
\end{minipage}
\begin{minipage}{0.15\textwidth}
\begin{tikzpicture}
\coordinate (1) at (-1,-1);
\coordinate (2) at (1,-1);
\coordinate (3) at (1,1);
\coordinate (4) at (-1,1);
\coordinate (5) at (0,-1);
\coordinate (6) at (1,0);
\coordinate (7) at (0,1);
\coordinate (8) at (-1,0);
\draw(1)--(2)--(3)--(4)--(1);
\draw (5)--(7);
\draw (6)--(8);
\fill (1) circle (2pt);
\fill (2) circle (2pt);
\fill (3) circle (2pt);
\fill (4) circle (2pt);
\fill (5) circle (2pt);
\fill (6) circle (2pt);
\fill (7) circle (2pt);
\fill (8) circle (2pt);
\fill ($(5)!0.5!(7)$) circle (2pt);
\node at (-0.5,-0.5) {$T_3$};
\node at (0.5,-0.5) {$T_3$};
\node at (0.5,0.5) {$T_3$};
\node at (-0.5,0.5) {$T_3$};
\node at (0,-1.5) {red};
\end{tikzpicture}
\end{minipage}
\begin{minipage}{0.15\textwidth}
\begin{tikzpicture}
\coordinate (1) at (-1,-1);
\coordinate (2) at (1,-1);
\coordinate (3) at (1,1);
\coordinate (4) at (-1,1);
\draw (1)--(2)--(3)--(4)--(1);
\draw (4)--($(1)!0.5!(2)$)--(3);
\fill (1) circle (2pt);
\fill (2) circle (2pt);
\fill (3) circle (2pt);
\fill (4) circle (2pt);
\fill ($(1)!0.5!(2)$) circle (2pt);
\node at (0,0.3) {$T_1$};
\node at (-0.6,-0.5) {$T_3$};
\node at (0.6,-0.5) {$T_3$};
\node at (0,-1.5) {green 1};
\end{tikzpicture}
\end{minipage}
\begin{minipage}{0.15\textwidth}
\begin{tikzpicture}
\coordinate (1) at (-1,-1);
\coordinate (2) at (1,-1);
\coordinate (3) at (1,1);
\coordinate (4) at (-1,1);
\draw (1)--(2)--(3)--(4)--(1);
\draw (4)--($(1)!0.5!(2)$)--($(2)!0.5!(3)$)--(4);
\fill (1) circle (2pt);
\fill (2) circle (2pt);
\fill (3) circle (2pt);
\fill (4) circle (2pt);
\fill ($(1)!0.5!(2)$) circle (2pt);
\fill ($(2)!0.5!(3)$) circle (2pt);
\node at (-0.6,-0.5) {$T_3$};
\node at (0.6,0.6) {$T_3$};
\node at (0,0) {$T_2$};
\node at (0.7,-0.7) {$T_4$};
\node at (0,-1.5) {green 2};
\end{tikzpicture}
\end{minipage}
\begin{minipage}{0.15\textwidth}
\begin{tikzpicture}
\coordinate (1) at (-1,-1);
\coordinate (2) at (1,-1);
\coordinate (3) at (1,1);
\coordinate (4) at (-1,1);
\draw(1)--(2)--(3)--(4)--(1);
\draw ($(1)!0.5!(2)$)--($(3)!0.5!(4)$);
\fill (1) circle (2pt);
\fill (2) circle (2pt);
\fill (3) circle (2pt);
\fill (4) circle (2pt);
\fill ($(1)!0.5!(2)$) circle (2pt);
\fill ($(3)!0.5!(4)$) circle (2pt);
\node at (-0.5,0) {$T_1$};
\node at (0.5,0) {$T_1$};
\node at (0,-1.5) {green 3};
\end{tikzpicture}
\end{minipage}
\caption{Area ratios for the refinement patterns in {\QRG}: $|T_0|/|T_0|=1,|T_0|/|T_1|=2,|T_0|/|T_2|=8/3,|T_0|/|T_3|=4$ and $|T_0|/|T_4|=8$.}
\label{fig:ratios}
\end{figure}
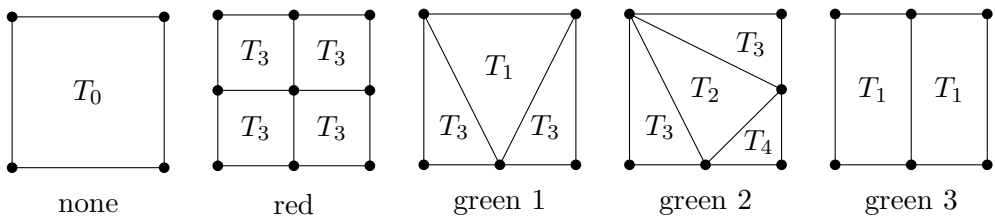

\subsection{Proof of $H^1$-Stability}

As was explained in \Cref{sec:prwork} and \Cref{rem:necessary}, $H^1$-stability relies
upon two competing effects: the localization properties of the
$L^2$-projection $\proj_h$, which in turn depend on the polynomial
degree of the FE space and the geometric shape of the elements;
and the rate at which the mesh size is allowed to
vary, which is reflected by the parameter $\mu$.

It was already observed in
\cite{BY, Siebert} for triangular meshes that the localization
properties of $\proj_h$ improve at first
for increasing polynomial degrees (up to $p=5$ or $p=6$),
after which they start deteriorating again.
This is also observed for quadrilateral meshes in \Cref{app:tables}.

For increasing $\mu$, the mesh size is allowed to vary more and this
naturally leads to deteriorating constants as well.

Finally,
in the presence of hanging nodes and due to the
continuity constraint $\V_h\subset C(\Omega)$,
the localization properties of $\proj_h$ deteriorate.
A piece-wise polynomial continuous function with a non-zero
value at one of the free nodes on the hanging edge (see Figure~\ref{fig:hangingnode})
has slightly larger support as it cannot be set to zero
at the hanging node without violating the continuity constraint.
This is reflected in Table~\ref{tab:quadrilateral1} and Table~\ref{tab:quadrilateral2} through deteriorating eigenvalues for $1$ and $2$ hanging nodes per element.

To show $H^1$-stability we require the following lemma.
\begin{lemma}[Minimal $\mu$]\label{lemma}
    For the refinement strategies {\QR}, {\QRG} and {\QRB},
    an initial mesh $\T_0$
    consisting of parallelograms
    and
    the choice $\mu=2$ for the weights from \Cref{def:siebert},
    we have $h_T\lesssim h_z$ for any $z\in\mc V(T)$,
    where the constant depends only on $\T_0$.
\end{lemma}

\begin{proof}
We closely follow the arguments in \cite{GaspozNVB,CCRGB}. We want to
show that there are constants $\mu>0$ and $C_\mu>0$ only dependent on the initial mesh such that
\begin{align}\label{eq:mu_ineq}
\gen(T')-\gen(T)\leq \mu \dist(z,z')+C_\mu\quad \forall T \in \mathcal{T}(z),T' \in \mathcal{T}(z'),
\end{align}
i.e., the difference in generations of two elements
is bounded by a multiple of the distance
and an additive constant,
where $\mathcal{T}(z):= \left\{T \in \mathcal{T}~|~z \in T\right\}$.
Then, \cref{eq:mu_ineq} would readily imply
$h_T \lesssim h_z$ for any $z \in \mc V(T)$.
This can be seen as follows.
For any $T \in \mathcal{T}$, there is a macro element $K_T \in \mathcal{T}_0$ such that $T \subset K_T$. For $z \in \mc V(T)$ let $T'$ and $z' \in T'$ be such that the weight $h_z$ satisfies
\begin{align*}
h_z^2 = 2^{\mu \dist(z,T')-\gen(T')}= 2^{\mu \dist(z,z')-\gen(T')}.
\end{align*}
Using \cref{eq:mu_ineq} and recalling the definition of $\gen(\cdot)$, we conclude
\begin{align*}
h_z^2 \geq 2^{-C_\mu-\gen(T)}=2^{-C_\mu}\frac{|T|}{|K_T|}\geq C h_T^2
\end{align*}
with $C:= 2^{-C_\mu}\frac{1}{|K_T|}$ only dependent on the initial mesh. This shows $h_T\lesssim h_z$ for any $z\in\mc V(T)$.

It thus remains to show \cref{eq:mu_ineq} with $\mu = 2$ and a suitable choice of $C_\mu $. 
Fix some $z,z'\in\mc V(\T)$ and
consider the shortest path of elements connecting $z$ and $z'$,
i.e., $$\CE(z,z'):=\left\{z=z_0,z_1,\ldots,z_{M-1},z_M = z'\right\},$$
where
$z_{i-1}$ and $z_i$ belong to the same element, $i=1,\ldots,M$,
and $M$ is minimal. We proceed in three steps.

(1) For any $z \in \mathcal{V}(T)$ we bound
\[\gen(T)-\gen(T') \leq \begin{cases} \alpha, & z \in \mathcal{V}(\T)\setminus\mathcal{V}(\T_0) \\ \alpha_0, & z \in \mathcal{V}(\T_0)\end{cases} \quad \text{ for all }T,T' \in \mathcal{T}(z).\]
Let $z \in \mathcal{V}(\T_0)$. An upper bound is given by
$\alpha_0 =\max\left\{\#\mathcal{T}_0(z_0)~|~z_0 \in \mathcal{V}(\T_0)\right\}$ for {\QR}, {\QRG} and {\QRB}.
The generation increases at most by 2
when crossing an edge of a macro element
and
we can traverse from the element with minimal generation to maximal generation by crossing at most $\lfloor\#\T_0(z_0)/2\rfloor$ macro elements.

%

Let $z \in \mathcal{V}(\T)\setminus\mathcal{V}(\T_0)$. The possible maximal differences of generations are shown in Figure~\ref{fig:maxdiff} and yield $\alpha = 2$ for {\QRB}, $\alpha = 3$ for {\QRG} and $\alpha=4$ for {\QR}.

(2) For any $z,z' \in \mathcal{V}(\T)\setminus \mathcal{V}(\T_0)$, we consider the maximal difference of generations for all elements $T \in \mathcal{T}(z)$, $T' \in \mathcal{T}(z')$,
where $\CE(z,z')\cap\mathcal{V}(\T_0) = \emptyset$.
By the above considerations,
a straight-forward upper bound would be
\[\gen(T)-\gen(T') \leq\alpha\#\CE(z,z') \quad\text{ for all } T \in \mathcal{T}(z), T' \in \mathcal{T}(z').\]
\Cref{fig:maxdiff} shows meshes yielding the maximal difference of generation. 
For {\QRB} this bound can not be improved because a sequence of scaled versions can be inserted into the most upper left quadrilateral, showing that this bound is optimal.

For {\QRG} we can improve this bound.
To this end, for $i=1,\ldots,M$, set
$$T_i = \arg \max\left\{\gen(T)~|~T \in \mathcal{T},\;
z_{i-1},z_i\in T\right\}.$$
Without loss of generality, we only consider generation increases. We see that the generation difference $3$ can only be attained once in a path
-- namely $\gen(T_1)=\gen(T)+3$ for some $T\in\T(z)$ --
and otherwise
$\gen(T_i) = \gen(T_{i-1})+2$ for $i=2,\ldots M$. We can thus conclude
the upper bound
\[\gen(T)-\gen(T') \leq 2\#\CE(z,z')+1 \quad\text{ for all } T \in \mathcal{T}(z), T' \in \mathcal{T}(z').\]
Note, that other situations than the one shown in \Cref{fig:maxdiff} can arise that also satisfy this reduced upper bound.

For {\QR}, improving the upper bound is a bit more involved, cf.~\cite[Proposition 3.8]{GaspozNVB}. Two reference situations for {\QR} are shown in \Cref{fig:maxdiff}. These show that in between two generation differences of 4, there must be one generation difference of 0, with any possible number of generation differences of 2 in between. Let $N_4$ denote the number of indices with generation difference 4 within $\CE(z,z')$. Using a telescopic sum this yields
\begin{align*}\gen(T)-\gen(T') &= \sum_{i=0}^M \gen(T_{i+1})-\gen(T_i)\\ &\leq 2\#\CE(z,z')+2N_4-2(N_4-1) =  2\#\CE(z,z')+2 .
\end{align*}
Thus, for all strategies, we have a common upper bound $2\#\CE(z,z')+2$.

(3) We now take a minimal path connecting $z$ and $z'$ and split the path into pieces $\CE(z,z')\setminus\mathcal{V}(\T_0)$ and $z \in \mathcal{V}(\T_0)$, i.\,e.,
\begin{align*}\CE(z,z') =   \CE(z,z_1) \cup \left\{z_1\right\} \cup \CE(z_1,z_2) \cup \cdots \cup \left\{z_J\right\} \cup \CE(z_J,z'), \end{align*}
where $J$ vertices from the initial mesh are contained in $\CE(z,z')$.
We can thus conclude that the bound for all $z,z'  \in \mathcal{V}(\T)$ is \begin{align*}\gen(T)-\gen(T') &\leq \sum_{j=0}^J  (2\#\CE(z_{j},z_{j+1})+2) + \sum_{j=1}^J \alpha_0 \\ &\leq 2\#\CE(z,z') +2(J+1)+J\alpha_0 \end{align*} for all $T \in \mathcal{T}(z),T' \in \mathcal{T}(z')$. Minimality of $\CE(z,z')$ gives $\CE(z,z')=\dist(z,z')+1$ and $J \leq \#\mathcal{V}(\T_0)$. This shows \cref{eq:mu_ineq} with $\mu = 2$ and $C_\mu = 4+\#\mathcal{V}(\T_0)(\alpha_0+1)$.
\end{proof}

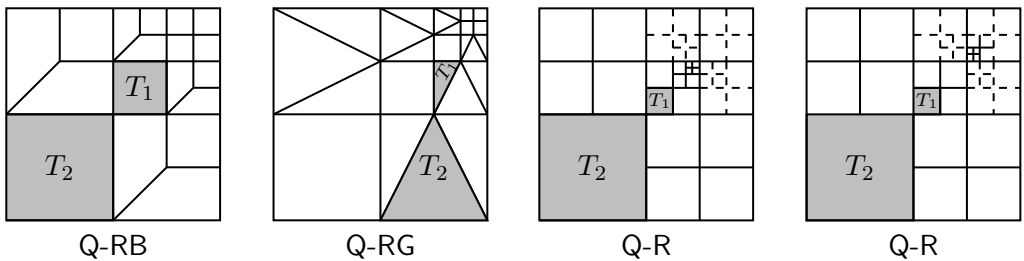
\begin{figure}

\tikzset{every picture/.style={line width=0.75pt}} 

\begin{tikzpicture}[x=0.8pt,y=0.8pt,yscale=-0.5,xscale=0.5]
\draw[fill=lightgray] (0,200)--(100,200)--(100,100)--(0,100);
\draw[fill=lightgray] (100,100)--(150,100)--(150,50)--(100,50);
\node at (50,150) {$T_2$};
\node at (125,75) {$T_1$};
\draw   (250,0) -- (450,0) -- (450,200) -- (250,200) -- cycle ;
\draw    (350,0) -- (350,200) ;
\draw    (450,100) -- (250,100) ;
\draw    (400,0) -- (400,100) ;
\draw    (350,50) -- (450,50) ;
\draw    (425,0) -- (425,50) ;
\draw    (400,25) -- (450,25) ;

\draw    (250,0) -- (350,50) ;
\draw    (250,100) -- (350,50) ;
\draw    (350,0) -- (400,25) ;
\draw    (350,50) -- (400,25) ;
\draw    (350,200) -- (400,100) ;
\draw    (400,100) -- (450,200) ;
\draw    (425,50) -- (450,100) ;
\draw    (400,100) -- (425,50) ;
\draw    (437,0) -- (437,25) ;
\draw    (425,12) -- (450,12) ;
\draw    (400,0) -- (425,12) ;
\draw    (400,25) -- (425,12) ;
\draw    (425,50) -- (437,25) ;
\draw    (437,25) -- (450,50) ;

\node at (100,225) {\QRB};

\draw[fill=lightgray] (350,200)--(400,100)--(450,200)--cycle;
\draw[fill=lightgray] (400,100)--(400,50)--(425,50) -- cycle;
\node at (400,150) {$T_2$};
\node[rotate=60] at (412,58) {\tiny $T_1$};
\draw   (0,0) -- (200,0) -- (200,200) -- (0,200) -- cycle ;
\draw    (100,0) -- (100,200) ;
\draw    (200,100) -- (0,100) ;
\draw    (150,0) -- (150,100) ;
\draw    (100,50) -- (200,50) ;
\draw    (175,0) -- (175,50) ;
\draw    (150,25) -- (200,25) ;

\draw    (50,0) -- (50,50) ;
\draw    (50,50) -- (100,50) ;
\draw    (0,100) -- (50,50) ;
\draw    (100,50) -- (125,25) ;
\draw    (125,0) -- (125,25) ;
\draw    (125,25) -- (150,25) ;
\draw    (175,50) -- (175,75) ;
\draw    (175,75) -- (200,75) ;
\draw    (175,75) -- (150,100) ;
\draw    (150,100) -- (150,150) ;
\draw    (150,150) -- (200,150) ;
\draw    (100,200) -- (150,150) ;
\node at (350,225) {\QRG};
\end{tikzpicture}\hspace*{15pt}
\begin{tikzpicture}[x=0.8pt,y=0.8pt,yscale=-0.5,xscale=0.5]
\draw[fill=lightgray] (0,450)--(100,450)--(100,350)--(0,350);
\draw[fill=lightgray] (100,350)--(125,350)--(125,325)--(100,325);
\node at (50,400) {$T_2$};
\node at (113,337) {\tiny $T_1$};
\draw   (0,250) -- (200,250) -- (200,450) -- (0,450) -- cycle ;
\draw    (100,250) -- (100,450) ;
\draw    (200,350) -- (0,350) ;
\draw    (150,250) -- (150,350) ;
\draw    (100,300) -- (200,300) ;
\draw[dashed]    (175,250) -- (175,300) ;
\draw[dashed]    (150,275) -- (200,275) ;
\draw[dashed]    (125,287) -- (150,287) ;
\draw[dashed]    (137,275) -- (137,300) ;
\draw[dashed]    (150,312) -- (175,312) ;
\draw[dashed]    (162,300) -- (162,325) ;
\draw    (0,300) -- (100,300) ;
\draw    (50,250) -- (50,350) ;
\draw    (150,350) -- (150,450) ;
\draw    (100,400) -- (200,400) ;
\draw[dashed]  (125,250) -- (125,300) ;
\draw    (125,300) -- (125,350) ;
\draw    (100,325) -- (150,325) ;
\draw[dashed]    (150,325) -- (200,325) ;
\draw[dashed]    (100,275) -- (150,275) ;
\draw[dashed]    (175,300) -- (175,350) ;

\draw    (137,300) -- (137,325) ;
\draw    (125,312) -- (150,312) ;
\draw    (137,306) -- (150,306) ;
\draw    (143,300) -- (143,312) ;

\node at (100,475) {\QR};

\draw[fill=lightgray] (250,350)--(250,450)--(350,450)--(350,350);
\draw[fill=lightgray] (350,350)--(375,350)--(375,325)--(350,325);
\node at (300,400) {$T_2$};
\node at (363,337) {\tiny $T_1$};
\draw   (250,250) -- (450,250) -- (450,450) -- (250,450) -- cycle ;
\draw    (350,250) -- (350,450) ;
\draw    (450,350) -- (250,350) ;
\draw    (400,250) -- (400,350) ;
\draw    (350,300) -- (450,300) ;
\draw[dashed]    (425,250) -- (425,300) ;
\draw[dashed]    (400,275) -- (450,275) ;

\draw    (250,300) -- (350,300) ;
\draw    (300,250) -- (300,350) ;
\draw    (400,350) -- (400,450) ;
\draw    (350,400) -- (450,400) ;
\draw[dashed]    (375,250) -- (375,300) ;
\draw   (375,300) -- (375,350) ;
\draw   (350,325) -- (400,325) ;
\draw[dashed]    (400,325) -- (450,325) ;
\draw[dashed]    (350,275) -- (400,275) ;
\draw[dashed]    (425,300) -- (425,350) ;

\draw    (400,287) -- (425,287) ;
\draw[dashed]    (375,287) -- (400,287) ;
\draw   (412,275) -- (412,300) ;
\draw[dashed]    (412,300) -- (412,325) ;
\draw[dashed]    (387,275) -- (387,300) ;
\draw[dashed]    (400,312) -- (425,312) ;
\draw    (406,287) -- (406,300) ;
\draw    (400,293) -- (412,293) ;
\node at (350,475) {\QR};
\end{tikzpicture}
\caption{Maximal differences of generations. $\gen(T_2)-\gen(T_1)=2$ for {\QRB}, $3$ for {\QRG} and $4$ for {\QR}.}
\label{fig:maxdiff}
\end{figure}

With this we can show
\begin{theorem}[$H^1$-Stability for {\QRG}]\label{thm:h1stab}
    For any initial mesh $\T_0$ consisting of parallelograms,
    {\QRG} is $H^1$-stable for polynomial degrees $p=2,\ldots,9$.
\end{theorem}

\begin{remark}[Hanging Nodes]\label{rem:hanging}
    The refinement strategy {\QR} 
    satisfies $h_T\lesssim h_z$ for
    $\mu=2$.
    However, due to hanging nodes, the localization
    properties of $\proj_h$ are too unfavorable such that
    $H^1$-stability can only be guaranteed for the restricted case
    $\mu=1$ and $p=1,2$.
\end{remark}

\subsection{General Quadrilaterals and Non-Linear Transformations}
\label{sec:nonlinear}

In \Cref{thm:h1stab} and \Cref{rem:hanging} we addressed meshes where
each element $T$ can be obtained by an affine transformation
$B_T:\hat T\rightarrow T$ onto $\hat T=[0,1]^2$
or $\hat T=\set{(\hat x,\hat y)^\intercal\in [0,1]^2:\;0\leq \hat x+\hat y\leq 1}$.
However, if we use the refinement strategy
{\QRB} from \cite{Funken2020} 
(see also \Cref{sec:review}) with an initial mesh consisting of parallelograms
-- such elements can only be obtained through a non-linear
transformation $B_T:\hat T\rightarrow T$.
The theory introduced above remains unchanged for this case.
However, the computation of the eigenvalue problem
in \cref{eq:evalue} is now more involved since it cannot be performed
entirely independent of $T$.

For a general quadrilateral $T\in\T$, we introduce the bilinear transformation
$B_T:[0,1]^2\rightarrow T$
\begin{align}\label{eq:bilinear}
    B_T
    \begin{pmatrix}
        \hat x\\
        \hat y
    \end{pmatrix}
    :=
    \begin{pmatrix}
        x_0\\
        y_0
    \end{pmatrix}
    +
    \begin{pmatrix}
        \cos(\theta) & -\sin(\theta)\\
        \sin(\theta)  &  \cos(\theta)
    \end{pmatrix}                
    \left[
    \begin{pmatrix}
         1 & s\\
         0 & 1
    \end{pmatrix}
    \begin{pmatrix}
        h_x & 0\\
        0     & h_y
    \end{pmatrix}
    \begin{pmatrix}
        \hat x\\
        \hat y
    \end{pmatrix}
    +
    \begin{pmatrix}
        \alpha\hat x\hat y\\
        \beta\hat x\hat y
    \end{pmatrix}
    \right],
\end{align}
where we used a convenient representation from \cite{QuadMesh}.
The free parameters $x_0$, $y_0$, $\theta$, $s$, $h_x$, $h_y$,
$\alpha$ and $\beta$ depend on $T$. Any general quadrilateral in 2D can be obtained
by using this transformation.

The parameters $(x_0,y_0)^\intercal\in\R^2$ describe translation of the reference point
$(0,0)^\intercal$. The parameter $\theta\in [0,2\pi)$ describes rotation.
The parameters $h_x,h_y>0$ describe stretching of the
horizontal side $(1, 0)^\intercal$ and the vertical side $(0, 1)^\intercal$,
respectively.

The parameter $s$ describes shearing of the square $[0,1]^2$. I.\,e.,
for $s=0$ (and $\alpha=\beta=0$), $T$ is a rectangle with side lengths
$h_x$ and $h_y$. The larger $s$, the larger the stretching of the rectangle to
a parallelogram. To maintain shape-regularity, the shearing $s$
has to remain bounded in absolute value.

Finally, $\alpha,\beta\in\R$ control the non-linear part
of the transformation. I.\,e., for $\alpha=\beta=0$, $T$ is a parallelogram.
Loosely speaking, for $\alpha>0$, $T$ becomes a trapezoid due to
shifting of the reference point $(1,1)^\intercal$ horizontally,
and similarly for $\beta>0$ vertically.

Computing the determinant of the Jacobian we obtain
\begin{align*}
    \left|\det DB_T
    \begin{pmatrix}
        \hat x\\
        \hat y
    \end{pmatrix}        
    \right|=
    h_xh_y+h_x\beta\hat x+h_y(\alpha-\beta s)\hat y.
\end{align*}
For $\alpha=\beta=0$, i.\,e., the affine case, this reduces to
$\left|\det DB_T\begin{pmatrix}\hat x\\\hat y\end{pmatrix}        
\right|=h_xh_y=|T|$. The new element mass matrix (without hanging
nodes) is
\begin{align*}
    M(T)=h_xh_yM_L+h_x\beta M_x+h_y(\alpha-\beta s)M_y
    \in\R^{\#\mc N(T)\times\#\mc N(T)},
\end{align*}
with
\begin{align*}
    (M_L)_{a,b}&:=\int_{[0,1]^2}\hat\varphi_a(\hat x,\hat y)
    \hat\varphi_b(\hat x,\hat y)\d\hat x\d\hat y,\\
    (M_x)_{a,b}&:=\int_{[0,1]^2}\hat\varphi_a(\hat x,\hat y)
    \hat\varphi_b(\hat x,\hat y)\hat x\d\hat x\d\hat y,\\
    (M_y)_{a,b}&:=\int_{[0,1]^2}\hat\varphi_a(\hat x,\hat y)
    \hat\varphi_b(\hat x,\hat y)\hat y\d\hat x\d\hat y,
\end{align*}
for $a,b\in\mc N(\hat T)$.
The matrix $A(T)$ from \cref{eq:evalue} is adjusted accordingly
to
\begin{align*}
    A(T)=h_xh_yA_L+h_x\beta A_x+h_y(\alpha-\beta s)A_y.
\end{align*}
Dividing both sides of \cref{eq:evalue} by $h_xh_y$, we obtain the
generalized eigenvalue problem
\begin{align*}
    \left(A_L+\frac{\beta}{h_y}A_x
    +\frac{\alpha-\beta s}{h_x}A_y\right)x=:\hat Ax=\lambda\hat Mx:=
    \lambda\left(M_L+\frac{\beta}{h_y}M_x+
    \frac{\alpha-\beta s}{h_x}M_y\right)x,\quad
    x\in\R^{\#\mc N(\hat T)}.
\end{align*}

We can now compute the smallest eigenvalue of the above problem
under some reasonable assumptions on the degree
of non-linearity.
E.\,g., for some constant $0<c<1$, assume
the non-linearity of $T$ is bounded as
\begin{align}\label{eq:nonlinear}
    |\alpha|\leq ch_x,\quad
    |\beta|\leq c\min\set{h_y,\frac{h_x}{s}}.
\end{align}
These assumptions were used in, e.\,g., \cite{QuadMesh}
(with $c=1/4$)
to show
$H^1$-stability of a Scott-Zhang type projector onto general
anisotropic quadrilateral meshes.
With this we obtain the bounds
\begin{align*}
    \frac{|\beta|}{h_y}\leq c,\quad
    \frac{|\alpha-\beta s|}{h_x}\leq 2c.
\end{align*}
Thus, for any given adaptive refinement strategy in 2D with general quadrilaterals,
$H^1$-stability follows from \eqref{eq:nonlinear}, Table~\ref{tab:nonlinear} and the following lemma.

\begin{lemma}[$H^1$-Stability General Quadrilaterals]
    Consider the generalized eigenvalue problem
    \begin{align}\label{eq:evnonlinear}
        \hat A x=\lambda\hat Mx,\quad x\in\R^n,\;n\in\N,
    \end{align}
    with
    \begin{align*}
        \hat A:=A_L+c_1A_x+c_2A_y,\quad
        \hat M:=M_L+c_1M_x+c_2M_y,\quad
        (c_1,c_2)^\intercal\in[-c, c]\times [-2c,2c],\;c>0,
    \end{align*}
    satisfying
    \begin{enumerate}[label=(\roman*)]
        \item    $\hat M$ is symmetric positive definite for any
                    $ (c_1,c_2)^\intercal\in[-c, c]\times [-2c,2c]$,
        \item    $A_L$, $A_x$ and $A_y$ are symmetric.
    \end{enumerate}
    
    Let $\lambda_{\min}(c_1,c_2)$ denote the smallest eigenvalue of
    \cref{eq:evnonlinear}. If the smallest eigenvalue satisfies
    \begin{align*}
        \inf_{\substack{-c\leq c_1\leq c,\\
        -2c\leq c_2\leq 2c}}\lambda_{\min}(c_1,c_2)
        \leq 0,
    \end{align*}
    then we have $\lambda_{\min}(c_1,c_2)\leq 0$ for some
    \begin{align*}
        (c_1,c_2)^\intercal\in\left\{
            \begin{pmatrix}
                -c\\
                -2c
            \end{pmatrix},
            \begin{pmatrix}
                c\\
                -2c
            \end{pmatrix},
            \begin{pmatrix}
                -c\\
                2c
            \end{pmatrix},
            \begin{pmatrix}
                c\\
                2c
            \end{pmatrix},
            \begin{pmatrix}
                0\\
                0
            \end{pmatrix}
        \right\}.
    \end{align*}
    In other words, to ensure all eigenvalues are strictly
    positive we only have to check 5
    combinations $(c_1,c_2)^\intercal$.
\end{lemma}

\begin{proof}
    The smallest eigenvalue of \cref{eq:evnonlinear} is characterized by
    minimizing the Rayleigh quotient
    \begin{align*}
        \lambda_{\min}(c_1,c_2)=\min_{x\in\R^N\setminus\{0\}}
        R_{c_1,c_2}(x):=
        \frac{\inp{x}{\hat Ax}}{\inp{x}{\hat Mx}}.
    \end{align*}
    The denominator is always positive since we assumed $\hat M>0$, i.\,e.,,
    $\hat M$ is symmetric positive definite.
    Thus, $\lambda_{\min}(c_1,c_2)\leq 0$ can only occur if
    $\inp{x}{\hat Ax}\leq 0$. For the latter we have
    \begin{align}\label{eq:combos}
        \inp{x}{\hat Ax}=
        \inp{x}{A_Lx}+
        c_1\inp{x}{A_xx}+
        c_2\inp{x}{A_yx}.
    \end{align}
    Each of the matrices $A_L$, $A_x$ and $A_y$ have real eigenvalues
    that can be either positive or negative. The statement of the lemma
    simply follows by considering the $2^3=8$ possibilities.
    
    E.g.,  if $A_L\leq 0$, then we can obtain
    $\lambda_{\min}(c_1,c_2)\leq 0$ for $c_1=c_2=0$.
    If $A_L>0$ and $A_x>0$, $A_y\leq 0$, then, if the sum
    in \Cref{eq:combos} is negative or zero
    for some $x\in\R^n\setminus\{0\}$,
    it will certainly hold for $c_1=-c$ and $c_2=2c$.
    Analogously for all other cases.
\end{proof}

The quadrilaterals of a blue pattern can be obtained by the bilinear transformation $B_T$ from \eqref{eq:bilinear} with $c_2\in \left\{-1/2,0,1\right\}$ and
$c_1=0$.
It thus follows from \Cref{tab:nonlinearRB} and \Cref{lemma}

\begin{theorem}[$H^1$-Stability for {\QRB}]\label{thm:stabqrb}
For any initial mesh $\T_0$ consisting of parallelograms,
{\QRB} is $H^1$-stable for polynomial degrees $p=2,\ldots,9$.
\end{theorem}

\bibliographystyle{acm}
\bibliography{literature}

\appendix
\newpage

\section{Tables}\label{app:tables}


\begin{table}[H]
\caption{Triangular element with no handing nodes.}
\label{tab:triangular}
\begin{tabular}{c||r}
\hline
\multicolumn{2}{c}{$\mu = 1$} \\ \hline \hline
$p$ & \multicolumn{1}{c}{$\lambda_\mathrm{min}$}\\ \hline
1 & $0.917136664350836$\\
2 &$0.951669246453928$ \\
3 & $0.964185270379702$ \\
4 & $0.968240261180467$\\
5 & $0.966806486169801$\\
6 & $0.959408791893528$\\
7 & $0.943501297553834$\\
8 & $0.913475197477444$\\
9 & $0.859176805106264$\\
10& $0.762190410650697$\\
11 & $0.589258437047241$\\
12 & $0.279511860034300$\\
\end{tabular}\hspace*{2cm}
\begin{tabular}{c||r}
\hline
\multicolumn{2}{c}{$\mu = 2$} \\ \hline \hline
$p$ & \multicolumn{1}{c}{$\lambda_\mathrm{min}$}\\ \hline
1 & $0.658493649053890$\\
2 & $0.800813482191006$\\
3 & $0.852396026216779$\\
4 &$0.869107942296914$\\
5 & $0.863198896177376$\\
6 & $0.832710628261139$\\
7 & $0.767150748223897$\\
8 & $0.643403571141314$\\
9 & $0.419622502039389$\\
10 &$0.019910501521261$ \\
11 & $-0.692797559342644$\\
12 & $-1.969362428534202$\\
\end{tabular}\\ \vspace*{2ex}
\begin{tabular}{c||r}
\hline
\multicolumn{2}{c}{$\mu = 3$} \\ \hline \hline
$p$ & \multicolumn{1}{c}{$\lambda_\mathrm{min}$}\\ \hline
1 & $0.192692294634035$\\
2 & $0.529130834062903$\\
3 &$0.651069958003112$ \\
4 & $0.690576276364528$\\
5 &$0.676607521609456$ \\
6 &$ 0.604534444362571$\\
7 & $0.449553442778550$ \\
8 & $0.157019938498792$\\
9 & $-0.371989788261829$\\
10& $-1.316893380981366$\\
11 &$ -3.001707463114409$ \\
12 & $-6.019457067017003$\\
\end{tabular}\hspace*{2cm}
\begin{tabular}{c||r}
\hline
\multicolumn{2}{c}{$\mu = 4$} \\ \hline \hline
$p$ & \multicolumn{1}{c}{$\lambda_\mathrm{min}$}\\ \hline
1 & $-0.536778579257494$\\
2 & $0.103660669859525$\\
3 & $0.335782117975507$\\
4 & $0.410985740336119$\\
5 & $0.384395032798194$\\
6 &$ 0.247197827175128$\\
7 & $-0.047821632992457$\\
8 & $-0.604683929864082$\\
9 & $-1.611698740822758$\\
10& $-3.410402743154268$\\
11 & $-6.617589017041926$ \\
12 & $-12.362130928403941$\\
\end{tabular}
\end{table}


\begin{table}
\caption{Parallelogram element with no hanging nodes. Note that for $\mu=2$ and $p=1$, $\lambda_{\min}$ is \emph{exactly} 0, i.e., this represents a ``border'' case.}
\label{tab:quadrilateral0}
\begin{tabular}{c||r}
\hline
\multicolumn{2}{c}{$\mu = 1$} \\ \hline \hline
$p$ & \multicolumn{1}{c}{$\lambda_\mathrm{min}$}\\ \hline
1 & $0.757359312880714$ \\
2 &  $0.909009742330268$ \\
3 & $0.944316491021463$\\
4 & $0.956248179064318$\\
5 &  $0.958418441371873$\\
6 &  $0.952907970706445$ \\
7 & $0.935573718078176$ \\
8 & $0.890974043187541$ \\
9 & $0.773040724644832$\\
10&$0.445576028516330$ \\
\end{tabular}\hspace*{2cm}
\begin{tabular}{c||r}
\hline
\multicolumn{2}{c}{$\mu = 2$} \\ \hline \hline
$p$ & \multicolumn{1}{c}{$\lambda_\mathrm{min}$}\\ \hline
1 &  $0.000000000000000$ \\
2 & $0.625000000000000$ \\
3 &  $0.770510421645976$\\
4 & $0.819684730309997$\\
5 & $0.828629076508982$ \\
6 &  $0.805918661652962$\\
7 & $0.734478653655678$ \\
8 & $0.550669106212760$ \\
9 & $0.064628121319203$\\
10& $-1.284958792632702$\\
\end{tabular}\\ \vspace*{2ex}
\begin{tabular}{c||r}
\hline
\multicolumn{2}{c}{$\mu = 3$} \\ \hline \hline
$p$ & \multicolumn{1}{c}{$\lambda_\mathrm{min}$}\\ \hline
1 &  $-1.363961030678928$ \\
2 & $0.113514613495402$\\
3 & $0.457495579824147$ \\
4 & $0.573741729216471$ \\
5 &  $0.594885815075774$\\
6 &  $0.541199279365591$\\
7 & $0.372317884428622$ \\
8 & $-0.062200722793164$\\
9 & $-1.211182670394353$\\
10& $-4.401553542490907$\\
\end{tabular}\hspace*{2cm}
\begin{tabular}{c||r}
\hline
\multicolumn{2}{c}{$\mu = 4$} \\ \hline \hline
$p$ & \multicolumn{1}{c}{$\lambda_\mathrm{min}$}\\ \hline
1 &  $-3.500000000000000$ \\
2 &  $-0.687500000000000$\\
3 & $-0.032703102593110$ \\
4 & $0.188581286394985$ \\
5 & $0.228830844290421$ \\
6 &  $0.126633977438330$\\
7 &  $-0.194846058549442$\\
8 & $-1.021989022042573$\\
9 & $-3.209173454063614$ \\
10& $-9.282314566847294$\\
\end{tabular}
\end{table}


\begin{table}
\caption{General quadrilateral element with
no hanging nodes and $c=1/4$, see \cref{eq:nonlinear}.}
\label{tab:nonlinear}
\begin{tabular}{c||r}
\hline
\multicolumn{2}{c}{$\mu = 1$} \\ \hline \hline
$p$ & \multicolumn{1}{c}{$\lambda_\mathrm{min}$}\\ \hline
1 & $0.747149409107802$ \\
2 &  $0.905648684583979$ \\
3 & $0.942338691490434$\\
4 & $0.954647079121367$\\
5 &  $0.956761324592774$\\
6 &  $0.950779093984766$ \\
7 & $0.932233847358632$ \\
8 & $0.884721376019483$ \\
9 & $0.759489811676522$\\
10&$0.412406997710229$ \\
\end{tabular}\hspace*{2cm}
\begin{tabular}{c||r}
\hline
\multicolumn{2}{c}{$\mu = 2$} \\ \hline \hline
$p$ & \multicolumn{1}{c}{$\lambda_\mathrm{min}$}\\ \hline
1 &  $-0.042078284125092$ \\
2 & $0.611148004334341$ \\
3 &  $0.762359276203259$\\
4 & $0.813086084543036$\\
5 & $0.821799567415631$ \\
6 &  $0.797144878710977$\\
7 & $0.720713976514372$ \\
8 & $0.524899861811526$ \\
9 & $0.008780468029065$\\
10& $-1.421658994070072$\\
\end{tabular}\\ \vspace*{2ex}
\begin{tabular}{c||r}
\hline
\multicolumn{2}{c}{$\mu = 3$} \\ \hline \hline
$p$ & \multicolumn{1}{c}{$\lambda_\mathrm{min}$}\\ \hline
1 &  $-1.463432454588482$ \\
2 & $0.080769035544651$\\
3 & $0.438226589642168$ \\
4 & $0.558142787768123$ \\
5 &  $0.578741121720422$\\
6 &  $0.520458398399101$\\
7 & $0.339778724066694$ \\
8 & $-0.123118212347723$\\
9 & $-1.343204346427102$\\
10& $-4.724707491574776$\\
\end{tabular}\hspace*{2cm}
\begin{tabular}{c||r}
\hline
\multicolumn{2}{c}{$\mu = 4$} \\ \hline \hline
$p$ & \multicolumn{1}{c}{$\lambda_\mathrm{min}$}\\ \hline
1 &  $-3.689352278562916$ \\
2 &  $-0.749833980495465$\\
3 & $-0.069383257085339$ \\
4 & $0.158887380443665$ \\
5 & $0.198098053370336$ \\
6 &  $0.087151954199392$\\
7 &  $-0.256787105685328$\\
8 & $-1.137950621848126$\\
9 & $-3.460487893869209$ \\
10& $-9.897465473315364$\\
\end{tabular}
\end{table}


\begin{table}
\caption{General quadrilateral element with no hanging nodes,
$c_1 =0$ and $c_2 \in \left\{-1/2,0,1\right\}$, see \cref{eq:evnonlinear}.}
\label{tab:nonlinearRB}
\begin{tabular}{c||r}
\hline
\multicolumn{2}{c}{$\mu = 1$} \\ \hline \hline
$p$ & \multicolumn{1}{c}{$\lambda_\mathrm{min}$}\\ \hline
1 & $0.752122198173689$ \\
2 &  $0.907368513818204$ \\
3 & $0.943368537006528$\\
4 & $0.955486626928499$\\
5 &  $0.957630538869594$\\
6 &  $0.951888510846911$ \\
7 & $0.933947667570817$ \\
8 & $0.887854856425997$ \\
9 & $0.766127960891231$\\
10&$0.428435362614456$ \\
\end{tabular}\hspace*{2cm}
\begin{tabular}{c||r}
\hline
\multicolumn{2}{c}{$\mu = 2$} \\ \hline \hline
$p$ & \multicolumn{1}{c}{$\lambda_\mathrm{min}$}\\ \hline
1 &  $-0.021583827383621$ \\
2 & $0.618235971544799$ \\
3 &  $0.766603599479458$\\
4 & $0.816546129999960$\\
5 & $0.825381877897596$ \\
6 &  $0.801717140994429$\\
7 & $0.727777178620049$ \\
8 & $0.537813938357047$ \\
9 & $0.036138407431271$\\
10& $-1.355600967716399$\\
\end{tabular}\\ \vspace*{2ex}
\begin{tabular}{c||r}
\hline
\multicolumn{2}{c}{$\mu = 3$} \\ \hline \hline
$p$ & \multicolumn{1}{c}{$\lambda_\mathrm{min}$}\\ \hline
1 &  $-1.414984357506708$ \\
2 & $0.097524713816906$\\
3 & $0.448260004468708$ \\
4 & $0.566322200392669$ \\
5 &  $0.587209564099582$\\
6 &  $0.531267048259225$\\
7 & $0.356475858596323$ \\
8 & $-0.092589838646902$\\
9 & $-1.278531243800609$\\
10& $-4.568548891511071$\\
\end{tabular}\hspace*{2cm}
\begin{tabular}{c||r}
\hline
\multicolumn{2}{c}{$\mu = 4$} \\ \hline \hline
$p$ & \multicolumn{1}{c}{$\lambda_\mathrm{min}$}\\ \hline
1 &  $-3.597127223226294$ \\
2 &  $-0.717938128048401$\\
3 & $-0.050283802342439$ \\
4 & $0.174457584999823$ \\
5 & $0.214218450539184$ \\
6 &  $0.107727134474930$\\
7 &  $-0.225002696209782$\\
8 & $-1.079837277393274$\\
9 & $-3.337377166559288$ \\
10& $-9.600204354723619$\\
\end{tabular}
\end{table}

\begin{table}
\caption{Parallelogram element with 1 hanging node.}
\label{tab:quadrilateral1}
\begin{tabular}{c||r}
\hline
\multicolumn{2}{c}{$\mu = 1$} \\ \hline \hline
$p$ & \multicolumn{1}{c}{$\lambda_\mathrm{min}$}\\ \hline
1 & $0.5713583436501$\\
2 & $0.5111112137989$\\
3 & $-0.4908310649897$\\
4 & $-5.0403752099103$\\
5 & $-28.0738045059095$ \\

\end{tabular}\hspace*{2cm}
\begin{tabular}{c||r}
\hline
\multicolumn{2}{c}{$\mu = 2$} \\ \hline \hline
$p$ & \multicolumn{1}{c}{$\lambda_\mathrm{min}$}\\ \hline
1 & $-0.7665695784117$\\
2 & $-1.0321608739122$\\
3 & $-5.1703397243886$\\
4 & $-23.9128701137892$\\
5 & $-118.8224619748803$\\

\end{tabular}\\ \vspace*{2ex}
\begin{tabular}{c||r}
\hline
\multicolumn{2}{c}{$\mu = 3$} \\ \hline \hline
$p$ & \multicolumn{1}{c}{$\lambda_\mathrm{min}$}\\ \hline
1 & $-3.1761016413483$\\
2 & $-3.8842461768688$ \\
3 & $ -13.7037380407129$\\
4 &$-57.9779496110135$ \\
5 & $-282.2556307086560$ \\

\end{tabular}\hspace*{2cm}
\begin{tabular}{c||r}
\hline
\multicolumn{2}{c}{$\mu = 4$} \\ \hline \hline
$p$ & \multicolumn{1}{c}{$\lambda_\mathrm{min}$}\\ \hline
1 & $-6.9495631028529$ \\
2 & $-8.5676118293778$ \\
3 & $-27.3679339412646$\\
4 & $-111.5515607551863$\\
5 & $-538.2010788870381$ \\

\end{tabular}
\end{table}

\begin{table}
\caption{Parallelogram element with 2 hanging nodes.}
\label{tab:quadrilateral2}
\begin{tabular}{c||r}
\hline
\multicolumn{2}{c}{$\mu = 1$} \\ \hline \hline
$p$ & \multicolumn{1}{c}{$\lambda_\mathrm{min}$}\\ \hline
1 & $0.3509421918476$\\
2 & $0.1968076252802$ \\
3 & $-1.3306311581102$ \\
4 & $-8.3677990882113$\\
5 & $-42.9093635066812$\\
\end{tabular}\hspace*{2cm}
\begin{tabular}{c||r}
\hline
\multicolumn{2}{c}{$\mu = 2$} \\ \hline \hline
$p$ & \multicolumn{1}{c}{$\lambda_\mathrm{min}$}\\ \hline
1 & $-1.6749751488843$ \\
2 & $-2.3102130737245$\\
3 & $-8.6052776052544$\\
4 & $-37.6077009566311$ \\
5 & $-179.9645530924988$\\
\end{tabular}\\ \vspace*{2ex}
\begin{tabular}{c||r}
\hline
\multicolumn{2}{c}{$\mu = 3$} \\ \hline \hline
$p$ & \multicolumn{1}{c}{$\lambda_\mathrm{min}$}\\ \hline
1 &  $-5.3235370099972$\\
2 & $-6.8252147095288$ \\
3 & $-21.7065019476733 $\\
4 & $-90.2671005455771$\\
5 & $-426.7931514447289$ \\
\end{tabular}\hspace*{2cm}
\begin{tabular}{c||r}
\hline
\multicolumn{2}{c}{$\mu = 4$} \\ \hline \hline
$p$ & \multicolumn{1}{c}{$\lambda_\mathrm{min}$}\\ \hline
1 &  $-11.0373881699796$\\
2 & $-13.8959588317600$ \\
3 & $-42.2237492236424$\\
4 & $-172.7346543048027$\\
5 & $-813.3404889159924$\\
\end{tabular}
\end{table}

\end{document}